\documentclass[reqno]{amsart} 
\usepackage{amscd,amssymb,latexsym,subfigure,verbatim,hyperref,epsfig,amsmath, xypic,supertabular}
\usepackage{tikz}
\usetikzlibrary{matrix, arrows}
\usepackage{tikz-cd}
\usepackage{enumitem}
\DeclareGraphicsExtensions{.pdf,.jpg, .png}
\usepackage{psfrag}
\usepackage{multicol}

\xyoption{all}
\usepackage[mathscr]{eucal}
\usepackage{mathrsfs}
\usepackage{fullpage}
\usepackage{graphicx}
\usepackage{hyperref}
\usepackage{todonotes}

\newcommand{\RR}{\mathbb{R}}

\newcommand{\ZZ}{\mathbb{Z}}

\newcommand{\calR}{\mathcal{R}}
\newcommand{\calJ}{\mathcal{J}}
\newcommand{\calP}{\mathcal{P}}
\newcommand{\calQ}{\mathcal{Q}}

\newcommand{\wDelta}{\hat{\Delta}}
\newcommand{\wtau}{\hat{\tau}}

\newcommand{\SMr}{C^{r}(\Delta)}

\newcommand{\reg}{\mathrm{reg}~}

\newcommand{\Initial}{\mathrm{In}}

\newcommand{\coker}{\mathrm{coker}~}
\newcommand{\img}{\mathrm{im}~}
\newcommand{\initdeg}{\mathrm{initdeg}}
\newcommand{\X}{\mathcal{X}} 
\newcommand{\ba}{\mathbf{a}} 

\theoremstyle{plain}
\newtheorem{Theorem}{Theorem}[section]
\newtheorem{Lemma}[Theorem]{Lemma}
\newtheorem{Proposition}[Theorem]{Proposition}
\newtheorem{Corollary}[Theorem]{Corollary}

\theoremstyle{definition}
\newtheorem{Definition}{Definition}[section]
\newtheorem{Example}{Example}[section]

\newtheorem{Problem}{Problem}
\newtheorem{KeyProblem}{Key Problem}
\newtheorem{Assumptions}{Assumptions}[section]

\theoremstyle{remark}
\newtheorem{Remark}{Remark}
\title{Planar splines on a triangulation with a single totally interior edge}
\author{Michael Dipasquale}
\address{DiPasquale: Department of Mathematics and Statistics\\ University of South Alabama\\ Mobile, AL\\ USA}
\email{mdipasquale@southalabama.edu}

\author{Beihui Yuan}
\address{Yuan: Department of Mathematics\\ Swansea university\\ Swansea SA1 8EN\\ UK}
\email{beihui.yuan@swansea.ac.uk}

\thanks{\textbf{Keywords}: Dimension of multivariate spline spaces, Gr\"obner bases, linear programming, lattice point enumeration}
\thanks{\textbf{2020 Mathematics Subject Classification}: Primary: 41A15, 13P25; Secondary: 13P10, 11P21, 52C05.}
\thanks{DiPasquale partially supported by NSF grant DMS-2201084}
\thanks{Yuan supported by EPSRC project EP/V012835/1}

\begin{document}
\begin{abstract}
We derive an explicit formula, valid for all integers $r,d\ge 0$, for the dimension of the vector space $C^r_d(\Delta)$ of piecewise polynomial functions continuously differentiable to order $r$ and whose constituents have degree at most $d$, where $\Delta$ is a planar triangulation that has a single totally interior edge.  This extends previous results of Toh\v{a}neanu, Min\'{a}\v{c}, and Sorokina.  Our result is a natural successor of Schumaker's 1979 dimension formula for splines on a planar vertex star.  Indeed, there has not been a dimension formula in this level of generality (valid for all integers $d,r\ge 0$ and any vertex coordinates) since Schumaker's result.  We derive our results using commutative algebra.
\end{abstract}
\maketitle

\section{Introduction}
Suppose $\Delta$ is a planar triangulation.  Given integers $r,d\ge 0$, we denote by $C^r_d(\Delta)$ the vector space of piecewise polynomial functions on $\Delta$ which are continuously differentiable of order $r$.  A fundamental problem in numerical analysis and computer-aided geometric design is to determine the dimension of (and a basis for) $C^r_d(\Delta)$~\cite{LaiSchumaker07}.  Even the dimension of the space $C^r_d(\Delta)$ turns out to be quite difficult.  It was known to Strang (see~\cite{billera1988homology}) that $\dim C^r_d(\Delta)$ depends on \textit{local geometry} (the number of slopes that meet at each interior vertex).  This dependence is standard by now, and is a fundamental part of a well-known lower bound for $\dim C^r_d(\Delta)$ derived by Schumaker~\cite{schumaker1979dimension}.  Hong proves in~\cite{Hong-1991} that Schumaker's lower bound coincides with $\dim C^r_d(\Delta)$ for $d\ge 3r+2$.  Under a mild genericity condition, Alfeld and Schumaker show that $\dim C^r_d(\Delta)$ also coincides with Schumaker's lower bound for $d=3r+1$~\cite{alfeld1990dimension}.

The `$2r+1$' conjecture of Schenck, appearing in his doctoral thesis~\cite{Schenck-1997-thesis} (see also~\cite{schenck2002cohomology,schenck2016algebraic}), is that $\dim C^r_d(\Delta)$ coincides with Schumaker's lower bound for $d\ge 2r+1$.  When $r\ge 2$, this has recently been disproved by the second author together with Schenck and Stillman~\cite{Yuan-Stillman-2019,Schenck-Stillman-Yuan-2020}.  If $r=1$, Schenck's conjecture reduces to a formula for the dimension of $C^1_3(\Delta)$ that has been conjectured since at least 1991 by Alfeld and Manni~\cite{schenck2016algebraic,Alfeld-2016-survey}.  This remains an open problem.

For $r+1\le d<3r+1$, there are relatively few general statements known about $\dim C^r_d(\Delta)$.  In this range it is possible that $\dim C^r_d(\Delta)$ depends upon \textit{global geometry} of $\Delta$ -- illustrated in the \textit{Morgan-Scott split}~\cite{Alfeld-2016-survey}.  The issue of geometric dependence can be sidestepped by assuming that $\Delta$ is suitably \textit{generic}.  
If $r=1$, Billera~\cite{billera1988homology} shows that the generic dimension of $\dim C^1_d(\Delta)$ coincides with Schumaker's lower bound for all $d\ge 2$, proving a conjecture of Strang~\cite{Strang73}.  
Whiteley computes certain generic dimension formulas for $r>1$ in~\cite{Whiteley-1991-Combinatorics}, but generic dimension formulas for $\dim C^r_d(\Delta)$ for all $d\ge r+1$ and $r\ge 2$ are (as of yet) out of reach.  Given the difficulty of computing the dimension of $C^r_d(\Delta)$ for $d<3r+1$, it is natural and useful to have a complete characterization of $\dim C^r_d(\Delta)$ for interesting examples, which is the motivation for our work.

In this paper we derive in Theorem~\ref{thm:ExplicitFormula} an explicit formula for $\dim C^r_d(\Delta)$ for all $r,d\ge 0$ whenever $\Delta$ is a triangulation that has a single totally interior edge -- that is, an edge connecting two interior vertices.  See Figure~\ref{fig:s3t4} for such a triangulation.  A more intuitive version of our main result, described in terms of lattice points in a certain polytope, appears in Theorem~\ref{prop:nontrivialcases}.  Our formula applies to any choice of vertex coordinates for $\Delta$, and only depends on the number of distinct slopes of edges meeting at each interior vertex.  This is the first non-trivial dimension formula for planar splines that applies in this level of generality (all $r,d\ge 0$ and any choice of vertex coordinates) since Schumaker computed the formula for splines on a planar vertex star in 1979~\cite{schumaker1979dimension}.

Our work directly extends results in previous papers of Toh\v{a}neanu, Min\'{a}\v{c}, and Sorokina~\cite{tohaneanu2005smooth,Minac-Tohaneanu-2013,sorokina2018} which study the dimension of splines on a particular triangulation with a single totally interior edge.  As a consequence of our work, we see that Schenck's `$2r+1$' conjecture is satisfied for triangulations with a single totally interior edge for all $r\ge 0$ (Corollary~\ref{cor:2r+1}).  Moreover, it is clear from our result that the dimension of splines on a triangulation with a single totally interior edge only depends on local geometry and not global geometry.  Thus the dependence of $\dim C^r_d(\Delta)$ upon global geometry indicated by the Morgan-Scott split does not manifest unless there is more than one totally interior edge.


We briefly outline the paper.  In Section~\ref{sec:background} we recall background on splines and dimension formulas from previous papers.  Section~\ref{sec:init} is a largely technical section in which we prove a few results in commutative algebra, possibly of independent interest, for use in future sections.  We then prove the first formulation of our main result -- Theorem~\ref{prop:nontrivialcases} -- in Section~\ref{setion:section_1tot}, stated in terms of lattice points.  In Section~\ref{sec:comparison} we characterize in what degrees the spline space does not change upon removal of the totally interior edge (this is related to the phenomenon of supersmoothness explored in~\cite{sorokina2018}).  We give the fully explicit dimension formula in Theorem~\ref{thm:ExplicitFormula} of Section~\ref{sec:explicitformula} and illustrate the result with several examples.  We conclude with additional remarks and open problems in Section~\ref{sec:conclusions}.

\section{Splines on planar triangulations}\label{sec:background}

We call a domain $\Omega\subset\RR^2$ \textit{polygonal} if it consists of a simple closed polygon and its interior.  The simple closed polygon is the boundary of $\Omega$, which we denote by $\partial\Omega$.  Throughout this paper we assume $\Delta$ is a triangulation of a polygonal domain; we denote the domain which $\Delta$ triangulates by $|\Delta|$.  For our purposes, a triangulation consists of a collection of triangles in which each pair $\sigma,\sigma'$ of triangles satisfies $\sigma\cap\sigma'=\emptyset$ or $\sigma\cap\sigma'$ is either an edge or vertex of both $\sigma$ and $\sigma'$.  Specifically, we do not allow so-called `hanging vertices' which occur in the interior of an edge of a triangle.

We write $\Delta_0$ for the set of vertices of $\Delta$, $\Delta_1$ for the set of edges of $\Delta$, and $\Delta_2$ for the set of triangles of $\Delta$.  An \textit{interior edge} of $\Delta$ is an edge that is a common edge of two triangles of $\Delta$.  A \textit{boundary edge} of $\Delta$ is an edge that is only contained in a single triangle of $\Delta$.  An \textit{interior vertex} of $\Delta$ is a vertex that is not contained in any boundary edges.  Put $\Delta^\circ_0$ and $\Delta^\circ_1$ for the set of interior vertices and interior edges of $\Delta$, respectively.  A \textit{totally interior edge} of $\Delta$ is an edge that connects two interior vertices of $\Delta$.

Let $r\geq0$ be an integer. We define $C^{r}(\Delta)$ to be the set of $C^{r}$-differentiable piecewise polynomial functions on $\Delta$. These functions are called \emph{splines}. More explicitly:

\begin{Definition}
$\SMr$ is the set of functions $F:\Delta\rightarrow \RR$ such that:
\begin{enumerate}
\item For all facets $\sigma\in\Delta$, $F_\sigma:=F|_{\sigma}$ is a polynomial in $\RR[x,y]$.
\item $F$ is differentiable of order $r$.
\end{enumerate}
\end{Definition}
\noindent For each integer $d\geq 0$, we define
\begin{align*}
C^{r}_{d}(\Delta):=\{F\in C^{r}(\Delta):\deg(F_\sigma)\leq d, \text{ for all }\sigma\in\Delta_2\}.
\end{align*}
\noindent For an edge $\tau\in\Delta_1$, we write $\tilde{\ell}_\tau$ for a choice of affine linear form that vanishes on the affine span of $\tau$.


\begin{Proposition}[Algebraic spline criterion]~\cite[Corollary~1.3]{billera1991dimension}\label{prop:algebraiccriterion}
Suppose $\Delta$ is a triangulation of a polygonal domain and $F:|\Delta|\to\RR$ is a piecewise polynomial function.  Then $F\in C^r(\Delta)$ if and only if
\[
\tilde{\ell}_\tau^{r+1}\mid F_{\sigma_1}-F_{\sigma_2}
\]
for every pair $\sigma_1,\sigma_2\in\Delta_2$ so that $\sigma_1\cap\sigma_2=\tau\in\Delta_1$.
\end{Proposition}

The space $C^r_d(\Delta)$ is a finite dimensional $\RR$-vector space. One of the key problems in spline theory is 
\begin{KeyProblem}\label{key_problem_dim}
Determine $\dim C^{r}_{d}(\Delta)$ for all $(\Delta,r,d)$.
\end{KeyProblem}


To study this problem, we use a standard coning construction due to Billera and Rose~\cite{billera1991dimension}.  Namely, for any set $U\subset\RR^2$, define $\hat{U}\subset \RR^3$ by $\hat{U}:=\{(sa,sb,s): (a,b)\in U \mbox{ and } 0\le s\le 1\}$.  Define $\wDelta$ to be the tetrahedral complex whose tetrahedra are $\{\widehat{\sigma}:\sigma\in\Delta_2\}$.  All above definitions for triangulations in $\RR^2$ carry over in the expected way to tetrahedral complexes in $\RR^3$.  For a two-dimensional face $\wtau\in\Delta_2$, we write $\ell_\tau$ for the linear form defining the linear span of $\wtau$ (this linear form is the homogenization of the affine linear form $\tilde{\ell}_\tau$).  We put 
\[
[C^r(\wDelta)]_d:=\{F\in C^r(\wDelta): F\in \RR[x,y,z]_d\},
\]
where $\RR[x,y,z]_d$ is the vector space of \textit{homogeneous} polynomials of degree $d$.  We relate this space to $C^r_d(\Delta)$ using:
\begin{Proposition}\cite[Theorem~2.6]{billera1991dimension}
The real vector spaces $[C^r(\wDelta)]_d$ and $C^r_d(\Delta)$ are isomorphic.
\end{Proposition}

Along with the coning construction, we use the Billera-Schenck-Stillman (BSS) chain complex over $\Delta$. This chain complex $\calR_{\bullet}/\calJ_{\bullet}$ is introduced by Billera in \cite{billera1988homology} and modified by Schenck and Stillman in \cite{schenck1997local}.

\begin{Definition}\label{def:faceideals}
For an edge $\tau\in\Delta^\circ_1$ and vertex $\gamma\in\Delta^\circ_0$, define
\begin{itemize}
\item $J(\tau):=\langle \ell_\tau^{r+1}\rangle$ (the principal ideal of $\RR[x,y,z]$ generated by $\ell_\tau^{r+1}$) and
\item $J(\gamma):=\sum_{\gamma\in\tau} J(\tau)$ (the ideal of $\RR[x,y,z]$ generated by $\{\ell^{r+1}_\tau:\tau\in\Delta^\circ_1\mbox{ and }\gamma\in\tau\}$.
\end{itemize}
\end{Definition}

Let $R=\RR[x,y,z]$.  We define the chain complex $\calR_\bullet$ by
\[
\calR_{\bullet}:=\bigoplus_{\sigma\in\Delta_2} R\xrightarrow{\partial_2} \bigoplus_{\tau\in\Delta^\circ_1} R \xrightarrow{\partial_1}  \bigoplus_{\tau\in\Delta^\circ_0} R,
\]
where the differentials $\partial_2$ and $\partial_1$ are the differentials in the simplicial chain complex of $\Delta$ relative to $\partial\Delta$ with coefficients in $R$.  In other words, the $i$th homology $H_i(\calR_\bullet)$ is isomorphic to $H_i(\Delta,\partial\Delta;R)$, where the latter is the $i$th simplicial homology group of $\Delta$ relative to $\partial\Delta$ with coefficients in $R$.  We also define the subcomplex $\calJ_{\bullet}\subset \calR_{\bullet}$ by
\[
0\xrightarrow{\partial_2}\bigoplus_{\tau\in\Delta^\circ_1}J(\tau)\xrightarrow{\partial_1}\bigoplus_{\gamma\in\Delta^{\circ}_0} J(\gamma)
\]
and the quotient complex (which we call the Billera-Schenck-Stillman chain complex)
\[
\calR_\bullet/\calJ_{\bullet}=\bigoplus_{\sigma\in\Delta_2} R\xrightarrow{\overline{\partial_2}} \bigoplus_{\tau\in\Delta^\circ_1} R/J(\tau) \xrightarrow{\overline{\partial_1}}\bigoplus_{\tau\in\Delta^\circ_0} R/J(\gamma).
\]
We use the following result of Schenck and Stillman.
\begin{Theorem}[Schenck and Stillman \cite{schenck1997family,schenck1997local}]\label{thm:splinedimension}
The dimension of $C^{r}_{d}(\Delta)$ is
\begin{equation}
\dim C^{r}_{d}(\Delta)=L(\Delta,r,d)+\dim H_{1}(\calR_{\bullet}/\calJ_{\bullet})_d,
\end{equation}
where $L(\Delta,r,d)$ is Schumaker's lower bound~\cite{schumaker1979dimension}.
\end{Theorem}

It is shown in~\cite{schenck1997family} that Schumaker's lower bound -- first derived using methods from numerical analysis in~\cite{schumaker1979dimension} -- is the Euler characteristic of $\calR_\bullet/\calJ_\bullet$ in degree $d$.  From this perspective, we give an explicit formula for $L(\Delta,d,r)$ for use later in the paper.  We set some additional notation.  If $A,B\in\ZZ$ are non-negative integers, we use the following convention for the binomial coefficient $\binom{A}{B}$:
\[
\dbinom{A}{B}=
\begin{cases}
\dfrac{A!}{B!(A-B)!} & B\le A\\
0 & \mbox{otherwise}
\end{cases}.
\]
For each vertex $\gamma\in\Delta^\circ_0$ we let $s_\gamma$ be the number of slopes of edges containing $\gamma$.  Let $\alpha_\gamma$ and $\nu_\gamma$ be the quotient and remainder when $s_\gamma(r+1)$ is divided by $s_\gamma-1$; that is, $s_\gamma(r+1)=\alpha_\gamma(s_\gamma-1)+\nu_\gamma$, with $\alpha_\gamma,\nu_\gamma\in\ZZ$ and $0\le \nu_\gamma<s_\gamma-1$.  Put $\mu_\gamma=s_\gamma-1-\nu_\gamma$.  Then we have the following formula for $L(\Delta,d,r)$.

\begin{Proposition}\label{prop:SchumakerLower}
Using the above notation, Schumaker's lower bound for $C^r_d(\Delta)$ can be expressed as
\[
L(\Delta,d,r)=\binom{d+2}{2}+\left(|\Delta^\circ_1|-\sum_{\gamma\in\Delta^\circ_0}s_\gamma\right)\binom{d+1-r}{2}+\sum_{\gamma\in\Delta^\circ_0} \left( \mu_\gamma\binom{d+2-\alpha_\gamma}{2}+\nu_\gamma\binom{d+1-\alpha_\gamma}{2} \right).
\]
\end{Proposition}

We will occasionally consider splines on a partition $\Delta$ which is not a triangulation but a \textit{rectilinear partition} -- in this case the polygonal domain $|\Delta|$ is subdivided into polygonal cells which meet along edges.  All definitions and results stated thus far carry over to rectilinear partitions.  The class of rectilinear partitions we will have occcasion to use are called \textit{quasi-cross-cut} partitions.

\begin{Definition}\label{def:quasi-cross-cut}
A rectilinear partition $\Delta$ is a quasi-cross-cut partition if every edge of $\Delta$ is connected to the boundary of $\Delta$ by a sequence of adjacent edges that all have the same slope.
\end{Definition}

The following result was first proved by Chui and Wang~\cite{Chui-Wang-1983}; we give the formulation of the result that appears in~\cite{schenck1997family}.

\begin{Proposition}\label{Prop:quasi-cross-cut}
If $\Delta$ is a quasi-cross-cut partition then
$
\dim C^r(\Delta)=L(\Delta,d,r)
$
for all $d,r\ge 0$.
\end{Proposition}

\subsection{Case of a single totally interior edge}\label{ss:one-interior-edge-setup}
In this section we specialize to the case of interest in this paper.  That is, $\Delta$ is a triangulation with only two interior vertices $v_1$ and $v_2$ connected by a single totally interior edge $\tau$.  
There are two cases in which the dimension formula on such a triangulation is trivial, which we record in the following proposition.

\begin{Proposition}\label{prop:trivialcases}
Let $\Delta$ be a triangulation with a single totally interior edge $\tau$ connecting interior vertices $v_1$ and $v_2$.  Suppose that either
\begin{itemize}
\item the interior edge $\tau$ has the same slope as another edge meeting $\tau$ at either $v_1$ or $v_2$ or
\item the number of slopes of edges meeting at either $v_1$ or $v_2$ is at least $r+3$.
\end{itemize}
Then $\dim C^r_d(\Delta)=L(\Delta,d,r)$ for all integers $d,r\ge 0$.
\end{Proposition}
\begin{proof}
The result follows from~\cite[Theorem~5.2]{schenck1997family}.  In either case, $H_1(\calR_\bullet/\calJ_\bullet)=0$ and $C^r(\wDelta)$ is a free module over the polynomial ring.
\end{proof}

\begin{Assumptions}\label{assumptions}
In the remainder of the paper, we use the following notation and assumptions whenever we have a triangulation $\Delta$ with a single totally interior edge $\tau$ connecting interior vertices $v_1$ and $v_2$.
\begin{itemize}
\item We assume no edge adjacent to $v_1$ or $v_2$ has the same slope as $\tau$.
\item We write $p$ (respectively $q$) for the number of edges \textit{different from} $\tau$ which are adjacent to $v_1$ (respectively $v_2$).
\item We write $s$ (respectively $t$) for the number of different slopes achieved by the edges \textit{different from }$\tau$ which contain $v_1$ (respectively $v_2$).
\item We assume (without loss) that $2\le s\le t\le r+1$.
\end{itemize}
\end{Assumptions}

\begin{Remark}
We explain the last bullet point in Assumptions~\ref{assumptions}. Since we assume no other edge besides $\tau$ has a slope equal to the slope of $\tau$, $v_1$ is surrounded by $p+1$ edges taking on $s+1$ different slopes and $v_2$ is surrounded by $q+1$ edges taking on $t+1$ different slopes.  See Figure~\ref{fig:s3t4}. We obtain $s\le t$ by simply relabeling $v_1$ and $v_2$ if necessary.  If $s=1$ then either $p=1$ or $p=2$.  In either case it is not possible for $\Delta$ to be a triangulation.  (If $p=2$ it would be possible to have a so-called $T$-juncture or `hanging vertex' at $v_1$, but we do not allow these under our definition of a triangulation.)  Hence $2\leq s, t$.  We can also assume that $s+1$ and $t+1$ are both at most $r+2$ by Proposition~\ref{prop:trivialcases}.  Putting these all together, we arrive at $2\le s\le t\le r+1$.
\end{Remark}



\begin{figure}
\begin{tikzpicture}
\tikzstyle{dot}=[circle,fill=black,inner sep=1 pt];

\def \epb {.5};
\def \eps {.25};

\node[dot] at (-1,0){};
\node[dot] at (1,0){};
\node[dot] at (2,1){};
\node[dot] at (1,2){};
\node[dot] at (0,1){};
\node[dot] at (-1,{2-\eps}){};
\node[dot] at (-2,{1}){};
\node[dot] at ({-2-\eps},{-1-\eps}){};
\node[dot] at ({-1},{-2+\epb}){};
\node[dot] at (0,-1){};
\node[dot] at ({1+\epb+\eps},-1){};

\draw[thick,black] (-1,0)node[left]{$v_1$}--node[above]{$\tau$}(1,0)node[right]{$v_2$};
\draw[thick,black] (1,0)--(2,1)--(1,2)--(0,1)--(-1,{2-\eps})--(-2,{1})--(-1,0);

\draw[thick,black] (-1,0)--({-2-\eps},{-1-\eps})--({-1},{-2+\epb})--(0,-1)--({1+\epb+\eps},-1)--(1,0);

\draw[thick,black] (1,0)--(1,2) (1,0)--(0,1) (1,0)--(0,-1);
\draw[thick,black] (-1,0)--(-1,{2-\eps}) (-1,0)--(0,1) (-1,0)--({-1},{-2+\epb}) (-1,0)--(0,-1);

\draw[thick,black] (-2,{1})--({-2-\eps},{-1-\eps});
\draw[thick,black] (2,1)--({1+\epb+\eps},-1);

\end{tikzpicture}
\caption{A triangulation with a single totally interior edge, $p=6$, $s=3$, $q=5$, and $t=4$.  A choice of coordinates that realizes this data is $v_1=(-1,0),v_2=(1,0)$ and, for the remaining vertices (read counterclockwise around the boundary, starting with the vertex northeast of $v_2$), $(2,1),(1,2),(0,1),(-1,7/4),(-2,1),(-9/4,-5/4),(-1,-3/2),(0,-1),$ and $(1,7/4)$.}
\label{fig:s3t4}
\end{figure}
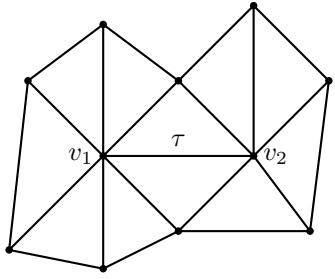

With this setup, consider the homology module $H_1(\calR_\bullet/\calJ_\bullet)$ where $\calR_\bullet/\calJ_\bullet$ is the Billera-Schenck-Stillman chain complex.  It turns out that $H_1(\calR_\bullet/\calJ_\bullet)$ is graded isomorphic to a shift of the quotient of the polynomial ring $R=\RR[x,y,z]$ by an ideal.

\begin{Lemma}\label{lem:presentation_H0J}
If $\Delta$ has only one totally interior edge $\tau$ , then
\begin{equation}
\label{eq:H0Presentation}
    H_{1}(\calR_{\bullet}/\calJ_{\bullet})\simeq R/(J_{1}:J(\tau)+J_{2}:J(\tau))(-r-1),
\end{equation}
where
\begin{equation}\label{eq:JiDef}
J_{i}=\sum_{\substack{\varepsilon\in\Delta^\circ_1\\v_{i}\in\varepsilon,\varepsilon\neq\tau}}J(\varepsilon) \mbox{ for } i=1,2.
\end{equation}
\end{Lemma}

This lemma is a consequence of presentation for $H_1(\calR_\bullet/\calJ_\bullet)$ due to Schenck and Stillman.  We recall this presentation before proceeding to the proof.

\begin{Lemma}\cite[Lemma~3.8]{schenck1997local}\label{lem:H0pres}
Let $\bigoplus_{\varepsilon\in \Delta^\circ_1} R[e_\varepsilon]$ be the free $R$ module with summands indexed by the formal basis symbols $\{[e_\varepsilon]\mid\varepsilon\in\Delta^\circ_1\}$ which each have degree $r+1$.  Define $K^r\subset \bigoplus_{\varepsilon\in \Delta^\circ_1} R[e_\varepsilon]$ to be the submodule of $\bigoplus_{\varepsilon\in \Delta^\circ_1} R[e_\varepsilon]$ generated by
\[
\{[e_\varepsilon]\mid\varepsilon\in\Delta^\circ_1 \mbox{ is not totally interior}\}
\]
and, for each $\gamma\in\Delta^\circ_0$,
\[
\left\lbrace \sum_{\gamma\in\varepsilon} a_\varepsilon[e_\varepsilon]\Bigm| \sum_{\gamma\in\varepsilon} a_\varepsilon\ell_\varepsilon^{r+1}=0\right\rbrace.
\]
The $R$-module $H_0(\calJ_\bullet)$ is given by generators and relations by
\[
0\rightarrow K^r\rightarrow \bigoplus_{\varepsilon\in \Delta^\circ_1} R[e_\varepsilon] \rightarrow H_0(\calJ_\bullet)\rightarrow 0. 
\]
\end{Lemma}

\begin{proof}[Proof of Lemma~\ref{lem:presentation_H0J}]
First, since $\Delta$ has no holes, $H_1(\calR_\bullet/\calJ_\bullet)\cong H_0(\calJ_\bullet)$.  This follows from the long exact sequence in homology associated to the short exact sequence of chain complexes $0\rightarrow \calJ_\bullet\rightarrow \calR_\bullet\rightarrow \calR_\bullet/\calJ_\bullet\rightarrow 0$ and the fact that $H_1(\calR_\bullet)=H_0(\calR_\bullet)=0$ (see~\cite{schenck1997local}).

Thus we may use Lemma~\ref{lem:H0pres}.  Since $\Delta$ has only one totally interior edge $\tau$, $K^r$ is generated by the free module $F=\{[e_\varepsilon]\mid \varepsilon\in\Delta^\circ_1,\varepsilon\neq \tau\}$ and the syzygy modules $K_i=\{\sum_{v_i\in\varepsilon} a_\varepsilon [e_\varepsilon]\bigm| \sum_{v_i\in\varepsilon} a_\varepsilon \ell_\varepsilon^{r+1}=0\}$ for $i=1,2$.  Since all factors of $R$ indexed by an interior edge different from $\tau$ are quotiented out, after trimming the presentation in Lemma~\ref{lem:H0pres} we are left with
\[
0\rightarrow \overline{K^r} \rightarrow R[e_{\tau}]\rightarrow H_0(\calJ_\bullet)\rightarrow 0,
\]
where $\overline{K^r}=K^r/F$.  Observe that $\overline{K^r}$ is the internal sum of the submodules
\[
K_i=\{ a_\tau[e_\tau] \mid a_\tau\ell^{r+1}_{\tau}\in J_i\}
\]
for $i=1,2$, where $J_{i}=\sum_{\substack{\varepsilon\in\Delta^\circ_1\\v_{i}\in\varepsilon,\varepsilon\neq\tau}}J(\varepsilon) \mbox{ for } i=1,2.$  Thus $\overline{K^r}=J_1:\ell^{r+1}_\tau+J_2:\ell^{r+1}_\tau=J_1:J(\tau)+J_2:J(\tau)$.  Recalling that $[e_\tau]$ has degree $r+1$, this proves that
\[
H_1(\calR_\bullet/\calJ_\bullet)\cong R/(J_1:J(\tau)+J_2:J(\tau))(-r-1).\qedhere
\]
\end{proof}

After coning, we apply a change of coordinates $T:\RR^3\to\RR^3$ so that $T(\hat{v_1})$ points in the direction of $(0,1,0)$ and $T(\hat{v_2})$ points in the direction of $(1,0,0)$.  With respect to this new choice of coordinates we may choose linear forms vanishing on the interior codimension one faces so that:
\begin{align*}
J(\tau)&=\langle z^{r+1}\rangle,\\
J_{1}&=\langle (x+{b_1}z)^{r+1},(x+{b_2}z)^{r+1},\dots,(x+b_{s}z)^{r+1}\rangle,\mbox{ and }\\
J_{2}&=\langle (y+{c_1}z)^{r+1},(y+{c_2}z)^{r+1},\dots,(y+c_{t}z)^{r+1}\rangle.
\end{align*}
In Section~\ref{sec:init} we study ideals of this type, returning to the study of the homology module in Section~\ref{setion:section_1tot}.

\section{The initial ideal of a power ideal in two variables}\label{sec:init}

This section is largely a technical section in which we derive some results from commutative algebra -- possibly of independent interest -- to use in our analysis for $\dim C^r_d(\Delta)$ in future sections.  The reader will not lose much by skipping this section for now and returning later as needed or desired.

Suppose we are given a set of points $\X=\{p_1,\cdots,p_s\}\subset \mathbb{P}^1$, where $p_i=[b_i:c_i]$ for $1\le i\le s$, and a sequence $\mathbf{a}=(a_1,\ldots,a_s)$ of \textit{multiplicities} for these points.  We will assume that the points are ordered so that $a_1\le a_2\le \cdots \le a_s$.  We associate two ideals to this set of points.  First, the \textit{power ideal}
\[
J(\X,\ba):=\langle (b_1x+c_1y)^{a_1+1},\ldots,(b_sx+c_sy)^{a_s+1}\rangle
\]
in the polynomial ring $R=\RR[x,y]$ (the offset by one in the exponent will make statements later a bit cleaner).  Secondly, the \textit{fat point} ideal
\[
I(\X,\ba):=\bigcap_{i=1}^s \langle b_iY-c_iX\rangle^{a_i} = \langle \prod_{i=1}^s (b_iY-c_iX)^{a_i}\rangle
\]
in the polynomial ring $S=\RR[X,Y]$ ($S$ is the coordinate ring of $\mathbb{P}^1$).  The ideal $I_{\ba}(\X)$ consists of all polynomials which vanish to order $a_i$ at $p_i$, for $i=1,\ldots,s$.

Our objective is to show that, under the assumption that $b_i\neq 0$ for $i=1,\ldots,s$, the initial ideal
$
\text{In}(J(\X,\ba)),
$
with respect to either graded lexicographic or graded reverse lexicographic order, is a lex-segment ideal.  Since the graded lexicographic and graded reverse lexicographic order coincide in two variables, we focus on the lexicographic order since it is consistent with the lex-segment definition.

\begin{Definition}\label{def:lexsegment}
A monomial ideal $I\subseteq R$ is called a lex-segment ideal if, whenever a monomial $m\in R$ of degree $d$ satisfies $m>_{\mbox{lex}} n$ for some monomial $n\in I$ of degree $d$, then $m\in I$.
\end{Definition}

Lex-segment ideals play an important role in Macaulay's classification of Hilbert functions~\cite{Macaulay-1927}.  Before proceeding to the proof, we introduce the notion of \textit{apolarity}.  An excellent survey of this notion by Geramita can be found in~\cite{geramita1995fat}.  Define an action of $S$ on $R$ by
\[
(X^aY^b)\circ f=\frac{\partial f}{\partial x^a\partial y^b},
\]
and extend linearly.  That is, $S$ acts on $R$ as \textit{partial differential operators}.  It is straightforward to see that this action induces a perfect pairing
\[
R_d\times S_d\to \RR
\]
via $(f,F)\to F\circ f$.  For an $\RR$-vector subspace $U\subset R_d$ we thus define
\[
U^\perp:=\{F\in S: F\circ f=0 \mbox{ for all }f\in U\}.
\]
Write $J_d(\X,\ba)$ for the $\RR$-vector space spanned by homogeneous polynomials in $J(\X,\ba)$ of degree $d$ (this definition clearly extends to any homogeneous ideal).  A result of Emsalem and Iarrobino describes $J_d(\X,\ba)^\perp$ in terms of fat point ideals.  In the statement of the result below, we put $[m]_+=\max\{m,0\}$ and $[d-\ba]_+=([d-a_1]_+,[d-a_2]_+,\ldots,[d-a_s]_+)$.

\begin{Theorem}[Emsalem and Iarrobino~\cite{emsalem1995inverse}]
\label{thm:InverseSystemSymbolicPower}
$J_d(\X,\ba)^{\perp}=I_d(\X,[d-\ba]_+)$
\end{Theorem}

As a corollary, the Hilbert function $\mbox{HF}(d,J(\X,\ba))=\dim J_d(\X,\ba)$ can be derived.

\begin{Corollary}[Geramita and Schenck~\cite{geramita1998fat}]
\label{cor:gerschenck}
$
\dim J_d(\X,\ba)=\min\left\lbrace d+1,\sum_{i=1}^s[d-a_i]_+\right\rbrace
$
\end{Corollary}

This shows that the Hilbert function of $J(\X,\ba)$ has the maximal growth possible for its number of generators.  We take this analysis one step further.

\begin{Corollary}
\label{cor:LexSeg1D}
Suppose that no point of $\X$ has a vanishing $x$-coordinate.  Then the initial ideal $
\text{In}(J(\X,\ba))
$
is a lex-segment ideal.
\end{Corollary}
\begin{proof}
Fix a degree $d$.  Put $F=\prod_{i=1}^s (b_iY-c_iX)^{[d-a_i]_+}$ and $\alpha=\deg(F)=\sum_{i=1}^s[d-a_i]_+$.  By assumption, $b_i\neq 0$ for any $i=1,\ldots,s$, so the monomial $Y^{\alpha}$ appears with non-zero coefficient in $F$.

Since $I(\X,[d-\ba]_+)$ is principle, a basis for $I_d(\X,[d-\ba]_+)$ is given by
\[
\{X^{d-\alpha-b}Y^bF: 0\le b\le d-\alpha\}
\]
(Coupled with Theorem~\ref{thm:InverseSystemSymbolicPower}, this proves that $\dim J_d(\X,\ba)=\alpha$, which is Corollary~\ref{cor:gerschenck}.)  Observe that the given basis for $I_d(\X,\ba)$ has a polynomial whose lex-last term involves the monomial $X^{d-\alpha-b}Y^{b+\alpha}$ for $0\le b\le d-\alpha$.

If $d<\min\{a_i\mid 1\le i\le s\}$ then $J_d(\X,\ba)=0$.  So suppose $d\ge\min\{a_i\mid 1\le i\le s\}$ and that the leading term of some polynomial $f\in J_d(\X,\ba)$ with respect to lex order is $Cx^{d-\alpha-b} y^{b+\alpha}$ for some $b\ge 0$ and $C\neq 0$.  Then every other term of $f$ involves a power of $y$ which is larger than $b+\alpha$.  From our above observation, the lex-last (or lex-least) monomial in the basis polynomial $X^{d-\alpha-b}Y^bF$ is $X^{d-\alpha+b} Y^{b+\alpha}$.  Thus $X^{d-\alpha-b}Y^bF\circ f\neq 0$.  In fact, $X^{d-\alpha-b}Y^bF\circ f=\prod_{i=1}^sb_iX^{d-\alpha-b}Y^{b+\alpha}\circ (Cx^{d-\alpha+b} y^{b+\alpha})$, so we can compute it exactly as:
\[
X^{d-\alpha-b}Y^bF\circ f=C(\prod_{i=1}^s b_i)(d-\alpha+b)!(b+\alpha)!,
\]
which is non-zero because the $b_i$'s are all non-vanishing and $C\neq 0$.  This contradicts Theorem~\ref{thm:InverseSystemSymbolicPower}, since $X^{d-\alpha-b}Y^bF\in I_d(\X,[d-\ba]_+)$ but $X^{d-\alpha-b}Y^bF\circ f\neq 0$.

It follows that the initial terms of $J_d(\X,\ba)$ can only involve the monomials $x^Ay^B$, where $0\le B<\alpha$.  Since $\dim J_d(\X,\ba)=\alpha$ by Corollary~\ref{cor:gerschenck}, it follows that $\mbox{In}(J(\X,\ba))_d$ consists of the $\alpha$ lex-largest monomials of degree $d$.  Thus $\mbox{In}(J(\X,\ba))$ is a lex-segment ideal.
\end{proof}

In the following corollary we use the ordering $a_1\le a_2\le\cdots\le a_s$.

\begin{Corollary}
\label{cor:GeneralMonomialDescription}
With the same setup as Corollary~\ref{cor:LexSeg1D},
The initial ideal $\mbox{In}(J(\X,\ba))$ consists of the monomials $x^Ay^B$, where $A\ge 0$, $B\ge 0$, and one of the strict inequalities $\sum_{i=1}^j a_i<jA+(j-1)B$, $1\le j\le s$, is satisfied.
\end{Corollary}
\begin{proof}
It suffices to show that $x^Ay^B\not\in \mbox{In}(J(\X,\ba))$ if and only if $A\ge 0, B\ge 0$ and $\sum_{i=1}^j a_i\ge jA+(j-1)B$ is satisfied for every $j=1,\ldots,s$.

Since $\mbox{In}(J(\X,\ba))$ is a lex-segment ideal with Hilbert function $\dim \mbox{In}(J(\X,\ba))_d=\min\{d+1,\sum_{i=1}^s[d-a_i]_+$, $x^Ay^B\not\in\mbox{In}(J(\X,\ba))$ if and only if
\[
B\ge \dim \mbox{In}(J(\X,\ba))_{A+B}=\min\left\lbrace A+B+1,\sum_{i=1}^s[A+B-a_i]_+\right\rbrace.
\]
Since $A,B\ge 0$ it is not possible that $B\ge A+B+1$.  So we are left with the condition
\[
B\ge\sum_{i=1}^s[A+B-a_i]_+.
\]
Now, since $a_1\le a_2\le\cdots\le a_s$, $A+B-a_1\ge A+B-a_2\ge \cdots A+B-a_s$.  The `plus' subscript means only positive contributions to the sum on the right hand side are taken.  So we can interpret the above inequality as
\[
B\ge\max\left\lbrace\sum_{i=1}^j (A+B-a_i): j=1,\ldots,s\right\rbrace.
\]
Equivalently, $B\ge \sum_{i=1}^j (A+B-a_i)$ is satisfied for $j=1,\ldots,s$.  Re-arranging, we get $x^Ay^B\not\in \mbox{In}(J(\X,\ba))$ if and only if $jA+(j-1)B\le \sum_{i=1}^j a_i$ for $i=1,\ldots,s$.
\end{proof}

\begin{Remark}
Given non-negative integers $a_1\le a_2\le\cdots\le a_s$, the inequalities $A\ge 0$, $B\ge 0$, and $jA+(j-1)B\le\sum_{i=1}^j a_i$ for $1\le j\le s$ define a convex polygon in $\RR^2$.  Corollary~\ref{cor:GeneralMonomialDescription} says that the initial ideal of $J(\X,\ba)$ consists of monomials which are in bijection with the lattice points in the first quadrant of $\RR^2$ and are additionally \textbf{not} contained in this polygon.  Equivalently, the monomials which are \textbf{not} in the initial ideal of $J(\X,\ba)$ are in bijection with the lattice points of this polygon.
\end{Remark}

In the next result, and following, if $r$ is a non-negative integer we write $J(\X,r)$ and $I(\X,r)$ for the case where $\ba=(r,r,\ldots,r)$ consists of $s$ copies of $r$.

\begin{Corollary}
\label{cor:equalValuedMonomialDescription}
With the same setup as Corollary~\ref{cor:LexSeg1D}, the initial ideal $\mbox{In}(J(\X,r))$ consists of those monomials $x^Ay^B$ satisfying $A\ge 0, B\ge 0$, and $sr<sA+(s-1)B$.
\end{Corollary}
\begin{proof}
Due to Corollary~\ref{cor:GeneralMonomialDescription}, it suffices to show that the inequality $sr<sA+(s-1)B$ is implied by the inequality $jr<jA+(j-1)B$ for any $j\le s$.  This is clear by multiplying both sides of $jr<jA+(j-1)B$ by $s/j$.
\end{proof}

\subsection{Behavior under colon}

In this section we discuss the behavior of $J(\X,\ba)$ under coloning with a power of $y$.  We continue to assume that no point of $\X$ has a vanishing $x$-coordinate.  We use the following fact about graded reverse lexicographic order.

\begin{Proposition}
\label{prop:commutingcolon}
If $I\subset R=\RR[x_1,\ldots,x_n]$ under graded reverse lexicographic order, then $\mbox{In}(I:x_n)=\mbox{In}(I):x_n$.  In particular, for any integer $e\ge 0$, $\mbox{In}(J(\X,\ba):y^e)=\mbox{In}(J(\X,\ba)):y^e$.
\end{Proposition}
\begin{proof}
This is a special case of~\cite[Proposition~15.12]{eisenbud2013commutative}.
\end{proof}

\begin{Corollary}
\label{cor:colontwovar}
For any integer $e\ge 0$, $\mbox{In}(J(\X,\ba):y^e)$ is a lex-segment ideal with Hilbert function
\[
\dim \mbox{In}(J(\X,\ba):y^e)_d=[\dim \mbox{In}(J(\X,\ba))_{d+e}-e]_+=[\min\{d+1, \sum_{i=1}^s [d+e-a_i]_+\}-e]_+
\]
The monomial $x^Ay^B$ is in $\mbox{In}(J(\X,\ba):y^e)$ if and only if $A\ge 0,B\ge 0$, and the inequality $\sum_{i=1}^j a_i-(j-1)e<jA+(j-1)B$ is satisfied for some $j=1,\ldots,s$.
\end{Corollary}
\begin{proof}
Due to Proposition~\ref{prop:commutingcolon} and the fact that graded lexicographic and graded reverse lexicographic orders coincide in two variables, we have $\mbox{In}(J(\X,\ba):y^e)=\mbox{In}(J(\X,\ba)):y^e$.  Now, a well-known identity is that $(\mbox{In}(J(\X,\ba)):y^e)y^e=\mbox{In}(J(\X,\ba))\cap\langle y^e\rangle$.  Said otherwise, the monomials in $\mbox{In}(J(\X,\ba)):y^e$ of degree $d$ are in bijection with the monomials of degree $d+e$ in $\mbox{In}(J(\X,\ba))$ which are divisible by $y^e$.  Since $(\mbox{In}(J(\X,\ba)))_{d+e}$ is spanned by lex-largest monomials, $\mbox{In}(J(\X,\ba)):y^e$ is either empty or consists of the $\dim(\mbox{In}(J(\X,\ba)))_{d+e}-e$ lex-largest monomials of degree $d$.  This establishes both that $\mbox{In}(J(\X,\ba)):y^e$ is lex-segment and the claimed form of the Hilbert function.

For the description of the monomials $x^Ay^B$ which are in $\mbox{In}(J(\X,\ba):y^e)=\mbox{In}(J(\X,\ba)):y^e$, it suffices to observe that $x^Ay^B\in \mbox{In}(J(\X,\ba)):y^e$ if and only if $x^Ay^{B+e}\in \mbox{In}(J(\X,\ba))$.  Then apply Corollary~\ref{cor:GeneralMonomialDescription}.
\end{proof}


\subsection{A summation property of Gr\"{o}bner bases}
We would like to prove a general fact, which will be useful in later sections. We refer the reader to~\cite[Chapter~2]{cox2015iva} for basics on Gr\"obner bases and the Buchberger algorithm, and we follow the same notation.
\begin{Lemma}\label{lemma:general_fact_revlex}
Let $R$ be the polynomial ring $\RR[x,y,z]$. Assume $I$ is a homogeneous ideal generated by polynomials in the variables $x$ and $z$ and $J$ is a homogeneous ideal generated by polynomials in the variables $y$ and $z$, then a Gr\"obner basis for $I+J$ with respect to graded lexicographic (or graded reverse lexicographic) order can be obtained by taking the union of the Gr\"obner bases of $I$ and $J$ with respect to the graded lexicographic (or graded reverse lexicographic) order.  In particular, $\Initial(I+J)=\Initial(I)+\Initial(J)$.
\end{Lemma}
\begin{proof}
Let $\mathcal{G}_1$ be a Gr\"obner basis for $I$ and $\mathcal{G}_2$ be a Gr\"obner basis for $J$, both taken with respect to either graded lexicographic order or graded reverse lexicographic order.  It suffices to show that $\mathcal{G}=\mathcal{G}_1\cup\mathcal{G}_2$ satisfies Buchberger's criterion - that is, the $S$-pair $S(f,g)$ of any two $f,g\in \mathcal{G}$ reduces to zero under the division algorithm.  This is clearly true if both $f$ and $g$ are in $\mathcal{G}_1$ or both $f$ and $g$ are in $\mathcal{G}_2$.  So we assume $f\in\mathcal{G}_1, g\in\mathcal{G}_2$.  We further assume the leading coefficients of $f$ and $g$ are normalized to $1$.  Let $\mbox{\textsc{LT}}(f)=x^Az^C$ and $\mbox{\textsc{LT}}(g)=y^Bz^D$.  Then
\[
\begin{array}{c}
f=x^Az^C+\mbox{terms in } x,z \mbox{ divisible by }z^C\\
g=y^Bz^D+\mbox{terms in } y,z \mbox{ divisible by }z^D
\end{array}.
\]
Put $f'=f-x^Az^C$ and $g'=g-y^Bz^D$.  Assume $C\ge D$ (the case $D\ge C$ is entirely analogous).  Then
\[
S(f,g)=\frac{g-g'}{z^D}f-\frac{f-f'}{z^D}g=\frac{f'}{z^D}g-\frac{g'}{z^D}f,
\]
where $\frac{f'}{z^D}$ and $\frac{g'}{z^D}$ are both polynomials because every term of $f'$ is divisible by $z^C$ (and hence $z^D$ since $C\ge D$) and every term of $g'$ is divisible by $z^D$.  There is no cancellation between the lead terms of $f'g/z^D$ and $g'f/z^D$ since the lead term of $f'g/z^D$ has a higher power of $y$ in it than $fg'/z^D$.  Thus $\textsc{LT}(S(f,g))=\max\left\lbrace\mbox{\textsc{LT}}\left(\frac{f'}{z^D}g\right), \mbox{\textsc{LT}}\left(\frac{g'}{z^D}f\right)\right\rbrace$.  Since $\mbox{\textsc{LT}}\left(\frac{f'}{z^D}g\right)\le \mbox{\textsc{LT}}(S(f,g))$ and $\mbox{\textsc{LT}}\left(\frac{g'}{z^D}f\right)\le \mbox{\textsc{LT}}(S(f,g))$,
\[
S(f,g)=\frac{f'}{z^D}g-\frac{g'}{z^D}f
\]
is what is called a \textit{standard representation} of $S(f,g)$ in~\cite[Section~9]{cox2015iva}.  It is shown in~\cite[Section~9]{cox2015iva} that if every $S$-pair of $\mathcal{G}$ has a standard representation, then $\mathcal{G}$ is a Gr\"obner basis, and so the result follows.
\end{proof}

\section{The dimension formula expressed via lattice points}\label{setion:section_1tot}

In this section we prove our first version of the dimension formula for $C^r_d(\Delta)$ when $\Delta$ is a triangulation with a single totally interior edge.  We also characterize when $\dim C^r_d(\Delta)$ begins to agree with Schumaker's lower bound.  We record these as two separate results, and prove them at the very end of the section.

\begin{Theorem}\label{prop:nontrivialcases}
Let $\Delta$ be a triangulation with a single totally interior edge $\tau$ satisfying Assumptions~\ref{assumptions}.  Then for all integers $d\ge 0$,
\begin{equation*}
\dim C^r_d(\Delta) =L(\Delta,d,r)+\#(\calP\cap\ZZ^3\cap H_d)
\end{equation*}
where $H_d=\{(A,B,C)\in \RR^3:A+B+C=d-r-1\}$ and $\calP$ is the polytope in $\RR^3$ defined by $A,B,C\ge 0$, $sA+(s-1)C\le r+1-s,$ and $tB+(t-1)C\le r+1-t$. 
 Equivalently, for all integers $d\ge 0$,
\[
\dim C^r_d(\Delta)=L(\Delta,d,r)+\#(\calP_d\cap\ZZ^2),
\]
where $\calP_d$ is the polygon in $\RR^2$ defined by the inequalities $A\ge 0, B\ge 0$, $A-B(s-1)\le sr-d(s-1)$, $B-A(t-1)\le tr-d(t-1)$, and $A+B\le d-r-1$.
\end{Theorem}

\begin{Theorem}\label{thm:LBregularity}
Let $\Delta$ be a triangulation with a single totally interior edge $\tau$ satisfying Assumptions~\ref{assumptions}.  If $r+1\equiv s-1\mod s$ and $r+1\equiv t-1\mod t$ then
	\begin{align*}
		\dim C^r_d(\Delta)> & L(\Delta,d,r) \mbox{ for }  d=\left\lfloor\frac{r+1}{s}\right\rfloor+\left\lfloor\frac{r+1}{t}\right\rfloor+r, \mbox{and}\\
		\dim C^r_d(\Delta)= & L(\Delta,d,r) \mbox{ for } d\ge\left\lfloor\frac{r+1}{s}\right\rfloor+\left\lfloor\frac{r+1}{t}\right\rfloor+r+1.
	\end{align*}
	Otherwise,
	\begin{align*}
		\dim C^r_d(\Delta)> & L(\Delta,d,r) \mbox{ for }  d=\left\lfloor\frac{r+1}{s}\right\rfloor+\left\lfloor\frac{r+1}{t}\right\rfloor+r-1\mbox{ and }\\
		\dim C^r_d(\Delta)= & L(\Delta,d,r) \mbox{ for } d\ge\left\lfloor\frac{r+1}{s}\right\rfloor+\left\lfloor\frac{r+1}{t}\right\rfloor+r.
	\end{align*}
\end{Theorem}

\begin{Remark}
In case there are three slopes that meet at each endpoint of the interior edge $\tau$ (so $s=t=2$), Theorem~\ref{thm:LBregularity} yields that $\dim C^r_d(\Delta)=L(\Delta,d,r)$ for $d\ge 2r+1$, which recovers the main result of Toh\v{a}neanu and Min\'{a}\v{c} in~\cite{Minac-Tohaneanu-2013}.  We say more on this in Example~\ref{ex:s2t2}.
\end{Remark}


\begin{Corollary}\label{cor:2r+1}
If $\Delta$ has a single totally interior edge, then $\dim C^r_d(\Delta)=L(\Delta,r,d)$ for $d\ge 2r+1$, so $\Delta$ satisfies the `$2r+1$' conjecture of Schenck~\cite{Schenck-1997-thesis} (see also~\cite[Conjecture~2.1]{schenck2002cohomology}).  In particular, $\dim C^1_d(\Delta)=L(\Delta,r,d)$ for $d\ge 3$, so $\Delta$ satisfies the conjecture of Alfeld and Manni for $\dim C^1_3(\Delta)$ (see~\cite[Conjecture~3]{Alfeld-2016-survey} or~\cite{schenck2016algebraic}).
\end{Corollary}
\begin{proof}
This is immediate from Proposition~\ref{prop:trivialcases} and Theorem~\ref{thm:LBregularity}, coupled with the fact that $t\ge s\ge 2$.
\end{proof}

We shall use Theorem~\ref{thm:splinedimension} to prove Theorems~\ref{prop:nontrivialcases} and~\ref{thm:LBregularity}, hence we spend the remainder of this section analyzing the homology module $H_1(\calR_\bullet/\calJ_\bullet)$, where $\calR_\bullet/\calJ_\bullet$ is the Billera-Schenck-Stillman chain complex from Section~\ref{sec:background}.  We use Assumptions~\ref{assumptions} throughout this section.
As we observed in Section~\ref{ss:one-interior-edge-setup}, we may change coordinates so that
\begin{align*}
J(\tau)&=\langle z^{r+1}\rangle,\\
J_{1}&=\langle (x+{b_1}z)^{r+1},(x+{b_2}z)^{r+1},\dots,(x+b_{s}z)^{r+1}\rangle,\mbox{ and}\\
J_{2}&=\langle (y+{c_1}z)^{r+1},(y+{c_2}z)^{r+1},\dots,(y+c_{t}z)^{r+1}\rangle.
\end{align*}
Using Lemma~\ref{lem:presentation_H0J}, Proposition~\ref{prop:commutingcolon}, and Lemma \ref{lemma:general_fact_revlex}, we obtain the following corollary.
\begin{Corollary}
\label{cor:initsumEsuminit}
With the above definition of $J_{1}$, $J_{2}$ and $J(\tau)$,
\begin{equation}\label{eq:initial_colon_ideal}
\Initial (J_{i}:J(\tau))=\Initial(J_{i}):J(\tau), ~\mbox{for}~i=1,2.
\end{equation}
and
\begin{equation}
\label{eq:sum_of_initials}
\Initial(J_{1}:J(\tau)+J_{2}:J(\tau))=\Initial(J_{1}:J(\tau))+\Initial(J_{2}:J(\tau)),
\end{equation}
where the initial ideal is taken with respect to graded lexicographic order or graded reverse lexicographic order.
\end{Corollary}

\begin{proof}
The equation \eqref{eq:initial_colon_ideal} follows from Proposition \ref{prop:commutingcolon}. Because $J_{1}:J(\tau)$ is only generated in polynomials in $x$ and $z$, and $J_{2}:J(\tau)$ is only generated in polynomials in $y$ and $z$, we may apply Lemma \ref{lemma:general_fact_revlex} here and obtain \eqref{eq:sum_of_initials}.
\end{proof}

\begin{Lemma}
\label{lem:quotientbasis}
A basis for $R/(J_1:J(\tau)+J_2:J(\tau))$ as an $\RR$-vector space is given by the monomials $x^Ay^Bz^C$ which satisfy the inequalities $A\ge 0, B\ge 0, C\ge 0, r+1-s\ge sA+(s-1)C,$ and $r+1-t\ge tB+(t-1)C$.
\end{Lemma}

\begin{proof}
A common use of initial ideals is that the monomials outside of $\Initial(I)$ form a basis for $R/I$~\cite[Section~5.3]{cox2015iva}.  Thus it suffices to show that $x^Ay^Bz^C\not\in\Initial(J_1:J(\tau)+J_2:J(\tau))$ if and only if $A,B,C$ satisfy the claimed inequalities.  Since $\Initial(J_1:J(\tau)+J_2:J(\tau))=\Initial(J_1:J(\tau))+\Initial(J_2:J(\tau))$ by~\eqref{eq:sum_of_initials}, it suffices to show that $x^Ay^Bz^C\notin \Initial(J_1:J(\tau))$ and $x^Ay^Bz^C\notin \Initial(J_2:J(\tau))$ if and only if the claimed inequalities hold.  Since the initial ideals are monomial, $x^Ay^Bz^C\notin \Initial(J_1:J(\tau))\iff x^Az^C\notin \Initial(J_1:J(\tau))$ and $x^Ay^Bz^C\notin \Initial(J_2:J(\tau))\iff y^Bz^C\notin \Initial(J_2:J(\tau))$.  Thus we reduce in both cases to two variables, and the result now follows from Corollary~\ref{cor:colontwovar}.
\end{proof}

\begin{Example}
Let $\Delta$ be the triangulation in Figure~\ref{fig:s3t4}, with $p=6, s=3, q=5, $ and $t=4$.  When $r=8$, a basis for $R/(J_1:J(\tau)+J_2:J(\tau))$ as an $\RR$-vector space is given by the monomials $x^Ay^Bz^C$ which satisfy $A\ge 0, B\ge 0,C\ge 0, 3A+2C\le 6,$ and $4B+3C\le 5$.  The lattice points $(A,B,C)\in\ZZ^3$ satisfying these inequalities are shown in Figure~\ref{fig:3Dpolytope}.  When $A+B+C=3$, there is a single lattice point -- $(2,1,0)$ -- that satisfies these inequalites (see the plot at right in Figure~\ref{fig:3Dpolytope}).  Thus $\dim (R/(J_1:J(\tau)+J_2:J(\tau)))_3=1$ and $\dim H_1(R_\bullet/J_\bullet)_{r+1+3}=\dim H_1(R_\bullet/J_\bullet)_{12}=1$.
\begin{figure}
\includegraphics[scale=.3]{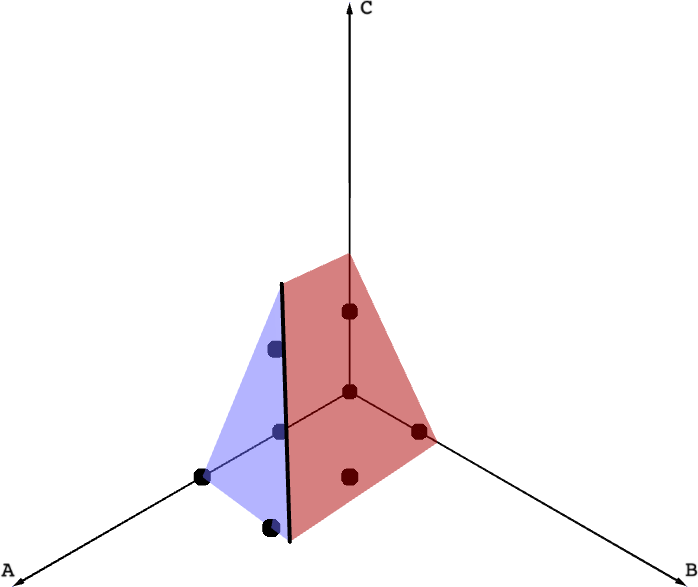}
\includegraphics[scale=.3]{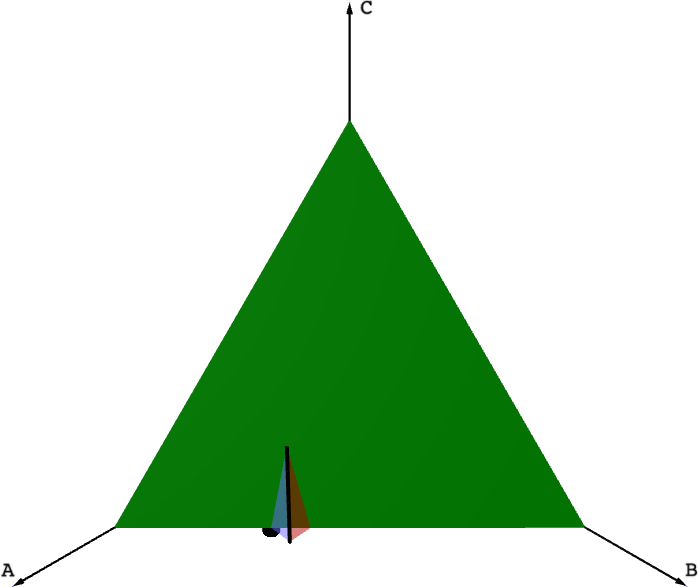}
\caption{Exponent vectors of monomials outside $R/(J_1:J(\tau)+J_2:J(\tau))$ when $p=6,s=3,q=5,t=4, $ and $r=8$.  The blue shaded plane (parallel to the $B$-axis) has equation $3A+2C=6$, while the red shaded plane (parallel to the $A$-axis) has equation $4B+3C=5$.  The picture at right shows the slice of this polytope when $A+B+C=3$. 
 There is one lattice point of the polytope in this slice.}\label{fig:3Dpolytope}
\end{figure}
\end{Example}

\begin{Proposition}\label{prop:2Dpolygon}
The dimension of $H_1(\calR_{\bullet}/J_\bullet)$ in degree $d$ is given by the number of lattice points $(A,B)\in\ZZ^2$ satisfying the inequalities $A\ge 0, B\ge 0$, $A-B(s-1)\le sr-d(s-1)$, $B-A(t-1)\le tr-d(t-1)$, and $A+B\le d-r-1$.
\end{Proposition}
\begin{proof}
The graded isomorphism~\eqref{eq:H0Presentation} shows that the dimension of $H_1(\calR_{\bullet}/J_\bullet)$ in degree $d$ is equal to the dimension of $R/(J_1:J(\tau)+J_2:J(\tau))$ in degree $d-r-1$.  From Lemma~\ref{lem:quotientbasis}, this is the number of lattice points $(A,B,C)\in\ZZ_{\ge 0}^3$ satisfying $A+B+C=d-r-1$, $A\ge 0$, $B\ge 0$, $C\ge 0$, $sA+(s-1)C\le r-s+1$, and $tB+(t-1)C\le r-t+1$.  We get the result by substituting $C=d-r-1-A-B$ and simplifying the ensuing inequalities.
\end{proof}

If $M$ is a graded module of finite length, recall that the \textit{(Castelnuovo-Mumford) regularity} of $M$, written $\reg M$, is defined by $\reg M:=\max\{d\mid M_d\neq 0\}$.

\begin{Proposition}\label{prop:bounds_H0J}
The regularity of $H_{1}(\calR_{\bullet}/\calJ_{\bullet})$ is bounded by 
\begin{equation}\label{eqn:bound_H0J}
 \left\lfloor\frac{r+1}{s}\right\rfloor+\left\lfloor\frac{r+1}{t}\right\rfloor+r-1\leq \reg H_1(\calR_{\bullet}/\calJ_{\bullet})\leq\left\lfloor \frac{r+1}{s}+\frac{r+1}{t}\right\rfloor+r-1.
\end{equation}
More precisely,
\begin{equation}\label{eqn:exact_H0J}
\reg H_1(\calR_\bullet/\calJ_\bullet)=
\begin{cases}
\left\lfloor\frac{r+1}{s}\right\rfloor+\left\lfloor\frac{r+1}{t}\right\rfloor+r & \mbox{ if }r+1\equiv s-1\mod s \mbox{ and } r+1\equiv t-1\mod t\\
\left\lfloor\frac{r+1}{s}\right\rfloor+\left\lfloor\frac{r+1}{t}\right\rfloor+r-1 & \mbox{otherwise}.
\end{cases}
\end{equation}
Thus $\dim C^r_d(\Delta)=L(\Delta,r,d)$ for $d> \frac{r+1}{s}+\frac{r+1}{t}+r-1$, where $L(\Delta,r,d)$ is Schumaker's lower bound~\cite{schumaker1979dimension}.
\end{Proposition}
To prove Proposition \ref{prop:bounds_H0J}, we use the following lemma:
\begin{Lemma}\label{lem:boundC}
Assume $2\le s\le t\le r+1$.  Let $\calP$ be the polytope in $\RR^3$ defined by the inequalities $A\ge 0, B\ge 0, C\ge 0, r+1-s\ge sA+(s-1)C,$ and $r+1-t\ge tB+(t-1)C$. Let $H$ be the plane defined by $A+B+C=\left\lfloor\frac{r+1}{s}\right\rfloor+\left\lfloor\frac{r+1}{t}\right\rfloor-1$.  Then $\calP\cap H\cap \ZZ^3\neq \emptyset$ if and only if $r+1\equiv s-1\mod s$ and $r+1\equiv t-1\mod t$.  Moreover, if $r+1\equiv s-1\mod s$ and $r+1\equiv t-1\mod t$ then
\begin{itemize}
\item If $t\ge 3$ then $\calP\cap H\cap \ZZ^3=\{(\left\lfloor\frac{r+1}{s}\right\rfloor-1,\left\lfloor\frac{r+1}{t}\right\rfloor-1, 1)\}$
\item If $s=t=2$ then $\calP\cap H\cap \ZZ^3=\{(t,t,r-1-2t)\mid t=0,1,\ldots, r/2-1\}$
\end{itemize}

\end{Lemma}
\begin{proof}
We first treat the case $t\ge 3$ and $s\ge 2$.  Assume $P=(A_{0},B_{0},C_{0})\in \calP\cap H\cap\ZZ^{3}$. If $C_{0}=0$, then $A_{0}\leq \left\lfloor\frac{r+1}{s}\right\rfloor-1$ and $B_{0}\leq \left\lfloor\frac{r+1}{t}\right\rfloor-1$. Hence, $(A_{0},B_{0},0)\not\in H$, contradiction. Therefore, $C_{0}\geq 1$.

Next, we show that $C_{0}\leq 1$. Let $d_0=\left\lfloor\frac{r+1}{s}\right\rfloor+\left\lfloor\frac{r+1}{t}\right\rfloor-1$. Substituting $A_{0}=d_{0}-B_{0}-C_{0}$ to $sA_{0}+(s-1)C_{0}\leq r+1-s$, we know that $(B_{0},C_{0})\in \ZZ^{2}_{\geq 0}$ must satisfy $sB_{0}+C_{0}\geq s+sd_0-(r+1)$ and $tB_{0}+(t-1)C_{0}\leq r+1-t$. Eliminating $B_{0}$ and simplifying, we obtain 
\begin{equation*}
    \left(1-\frac{1}{s}-\frac{1}{t}\right)C_{0}\leq \left\{\frac{r+1}{s}\right\}+\left\{\frac{r+1}{t}\right\}-1.
\end{equation*}
where $\left\{\frac{r+1}{s}\right\}=\frac{r+1}{s}-\left\lfloor\frac{r+1}{s}\right\rfloor$ and $\left\{\frac{r+1}{t}\right\}=\frac{r+1}{t}-\left\lfloor\frac{r+1}{t}\right\rfloor$. Because $\left\{\frac{r+1}{s}\right\}\leq 1-\frac{1}{s}$ and $\left\{\frac{r+1}{t}\right\}\leq 1-\frac{1}{t}$, so $\left(1-\frac{1}{s}-\frac{1}{t}\right)C_{0}\leq 1-\frac{1}{s}-\frac{1}{t}$. Since $s\geq 2$ and $t\geq 3$, this implies $C_{0}\leq 1$.

Therefore, $C_{0}=1$, and hence $A_{0}\leq \left\lfloor\frac{r+2}{s}\right\rfloor-2\leq \left\lfloor\frac{r+1}{s}\right\rfloor-1$ and $B_{0}\leq \left\lfloor\frac{r+2}{t}\right\rfloor-2\leq \left\lfloor\frac{r+1}{t}\right\rfloor-1$.  Observe that if $\left\lfloor\frac{r+2}{s}\right\rfloor-2<\left\lfloor\frac{r+1}{s}\right\rfloor-1$ or $\left\lfloor\frac{r+2}{t}\right\rfloor-2<\left\lfloor\frac{r+1}{t}\right\rfloor-1$ then $H\cap\calP\cap\ZZ^3=\emptyset$.  Therefore if $H\cap\calP\cap\ZZ^3\neq\emptyset$ then $\left\lfloor\frac{r+2}{s}\right\rfloor=\lfloor \frac{r+1}{s}\rfloor+1$ and $\left\lfloor\frac{r+2}{t}\right\rfloor=\lfloor \frac{r+1}{t}\rfloor+1$ which in turn happens if and only if $r+1\equiv s-1\mod s$ and $r+1\equiv t-1\mod t$.  In case both congruences are satisfied, it is clear from the above reasoning that $(A_{0},B_{0},C_{0})=(\left\lfloor\frac{r+1}{s}\right\rfloor-1,\left\lfloor\frac{r+1}{t}\right\rfloor-1, 1)$ is the only point in $\calP\cap H\cap \ZZ^{3}$.

Now we treat the case $s=t=2$.  First suppose $r$ is odd, so $r=2k-1$ for some integer $k\ge 1$.  Then $d_0=2k-1$ and the polytope $\calP$ is defined by $A\ge 0$, $B\ge 0$, $C\ge 0$, $2A+B\le 2k-2$, and $2B+C\le 2k-2$.  From the final two inequalities we deduce that $A+B+C\le 2k-2$ and thus $H\cap\calP$ is empty.  Now suppose $r=2k$ for some integer $k\ge 1$.  Then $d_0=2k-1$ again, and $\calP$ is defined by the inequalities $A\ge 0$, $B\ge 0$, $C\ge 0$, $2A+B\le 2k-1$, and $2B+C\le 2k-1$.  From $A+B+C=2k-1$, $2A+B\le 2k-1$, and $2B+C\le 2k-1$, we deduce that $A=B$.  Since $A\ge 0$ and $B\ge 0$, we deduce that $H\cap\mathcal{P}\cap\ZZ^3=\{(t,t,r-1-2t\mid t=0,\ldots,r/2-1\}$.
\end{proof}
Now we are ready to prove Proposition \ref{prop:bounds_H0J}.
\begin{proof}[Proof of Proposition \ref{prop:bounds_H0J}]
We first prove the bounds~\eqref{eqn:bound_H0J}.  By Lemma~\ref{lem:presentation_H0J}, the regularity of $H_1(\calR_{\bullet}/\calJ_{\bullet})$ is the largest degree of a monomial in $R/(J_1:J(\tau)+J_2:J(\tau))(-r-1)$.  By Lemma \ref{lem:quotientbasis}, $x^{A}y^{B}\in R/(J_{1}:J(\tau)+J_{2}:J(\tau))$ if $r+1\geq s(A+1)$ and $r+1\geq t(B+1)$. In particular, $A=\left\lfloor\frac{r+1}{s}\right\rfloor-1$ and $B=\left\lfloor\frac{r+1}{t}\right\rfloor-1$ satisfy this condition. By Lemma \ref{lem:presentation_H0J}, we obtain 
\begin{equation*}
     \left\lfloor\frac{r+1}{s}\right\rfloor+\left\lfloor\frac{r+1}{t}\right\rfloor+r-1\leq \reg H_1(\calR_{\bullet}/\calJ_{\bullet}).
\end{equation*}
Note that the region of $\RR^{3}$ bounded by inequalities in Lemma \ref{lem:quotientbasis} is a polytope, which we denote by $\calP$ as in Lemma~\ref{lem:boundC}.  Thus the largest degree of a monomial in $R/(J_1:J(\tau)+J_2:J(\tau))$ is obtained by maximizing the linear functional $A+B+C$ over $\calP\cap\ZZ^3$.  It is well-known in linear programming that the maximum of this linear functional on $\calP\cap\RR^3$ occurs at one of the vertices of $\calP$.  Therefore, to prove the upper bound, it suffices to verify that evaluating $A+B+C$ at the vertices achieves a value of at most $\frac{r+1}{s}+\frac{r+1}{t}-2$. The vertices of $\calP$ have coordinates
\begin{align*}
    \left(\dfrac{t-s}{s(t-1)}r,0,\dfrac{r}{t-1}-1\right),~
    \left(\dfrac{r+1}{s}-1,\dfrac{r+1}{t}-1,0\right),\\
    \left(0,0,\dfrac{r}{t-1}-1\right),~
    \left(\dfrac{r+1}{s}-1,0,0\right),~
    \left(0,\dfrac{r+1}{t}-1,0\right),~
    (0,0,0),
\end{align*}
respectively. Computing $A+B+C$ for each of them, we have
\begin{align*}
    \frac{t}{s(t-1)}r-1, ~\frac{r+1}{s}+\frac{r+1}{t}-2,\\
    \frac{r}{t-1}-1,~\frac{r+1}{s}-1,~\frac{r+1}{t}-1,~0,
\end{align*}
respectively. We want to show that
$\frac{r+1}{s}+\frac{r+1}{t}-2$ is the largest among all of them.\\
It is clear that
\begin{equation*}
    0\leq \frac{r+1}{t}-1\leq\frac{r+1}{s}-1\leq \frac{r+1}{s}+\frac{r+1}{t}-2,
\end{equation*}
and that
\begin{equation*}
    \frac{r}{t-1}-1\leq \frac{t}{s(t-1)}r-1.
\end{equation*}
We only need to show that
\begin{equation}\label{eq:ineq_two_vertices_of_polytope}
    \frac{t}{s(t-1)}r-1\leq\frac{r+1}{s}+\frac{r+1}{t}-2.
\end{equation} 
We have
\begin{equation*}
    \frac{t}{s(t-1)}r-1-\left[\frac{r+1}{s}+\frac{r+1}{t}-2\right]=\left[\frac{1-(t-1)(s-1)}{t(t-1)s}\right]r+1-\frac{1}{s}-\frac{1}{t}.
\end{equation*}
Because $t\geq s\geq 2$, so $1-(t-1)(s-1)\leq 0$, where equality holds if and only if $t=s=2$. If $t=s=2$, then equality in \eqref{eq:ineq_two_vertices_of_polytope} holds. Otherwise, $1-(t-1)(s-1)< 0$ and
\begin{align*}
    \left[\frac{1-(t-1)(s-1)}{t(t-1)s}\right]r+1-\frac{1}{s}-\frac{1}{t}&\leq \left[\frac{1-(t-1)(s-1)}{t(t-1)s}\right](t-1)+1-\frac{1}{s}-\frac{1}{t}\\
    &=0.
\end{align*}
Thus, we have proved \eqref{eq:ineq_two_vertices_of_polytope}. 
This means 
\begin{equation*}
    \reg H_1(\calR_{\bullet}/\calJ_{\bullet})\leq\left\lfloor\frac{r+1}{s}+\frac{r+1}{t}\right\rfloor+r-1.
\end{equation*}
Therefore, the inequality \eqref{eqn:bound_H0J} holds.

Now we prove Equation~\eqref{eqn:exact_H0J}.  If $\reg H_{1}(\calR_{\bullet}/\calJ_{\bullet})\neq\left\lfloor\frac{r+1}{s}\right\rfloor+\left\lfloor\frac{r+1}{t}\right\rfloor+r-1$, then by \eqref{eqn:bound_H0J}, $\reg H_{1}(\calR_{\bullet}/\calJ_{\bullet})=\left\lfloor\frac{r+1}{s}\right\rfloor+\left\lfloor\frac{r+1}{t}\right\rfloor+r$. By Lemma \ref{lem:presentation_H0J} and Lemma \ref{lem:quotientbasis}, this is equivalent to saying that $\calP\cap H\cap \ZZ^3\neq \emptyset$, where $\calP$ and $H$ are as in the setup of Lemma \ref{lem:boundC}.  Applying Lemma~\ref{lem:boundC} completes the proof. \qedhere

\end{proof}
\begin{Example}
Assume that $(s,t)=(3,4)$ and $r=6$. Then
\begin{align*}
\Initial(J_{1})=\langle x^{7}, x^{6}z, x^{5}z^2,x^{4}z^4,x^{3}z^5,x^2z^{7},xz^{8},z^{10}\rangle
\end{align*}
and
\begin{equation*}
\Initial(J_{2})=\langle y^{7}, y^{6}z, y^5z^2,y^4z^3,y^3z^{5},y^2z^{6},yz^{7}, z^{9}\rangle.
\end{equation*}
Therefore,
\begin{equation*}
\Initial (J_{1}\colon J(\tau)+J_{2}\colon J(\tau))=\langle x^2, xz, y,z^2\rangle.
\end{equation*}
Clearly every monomial of degree two or more is in $\langle x^2, xz, y,z^2\rangle$, but the monomial $x$ is not in this ideal.  Therefore $\reg R/\langle x^2, xz, y,z^2\rangle=1$ and so $\reg H_1(\calR_\bullet/\calJ_\bullet)=8$ by Lemma~\ref{lem:presentation_H0J}.
The bounds given by \eqref{eqn:bound_H0J} are
\begin{equation*}
    8\leq \reg H_{1}(\calR_{\bullet}/\calJ_{\bullet})\leq 9.
\end{equation*}
In this case, $r+1\not\equiv s-1\mod s$, so by Proposition \ref{prop:bounds_H0J}, $\reg H_{1}(\calR_{\bullet}/\calJ_{\bullet})=8$, which aligns with what we have found already.\\
On the other hand, if $(s,t)=(3,4)$ and $r=10$, then
\begin{align*}
\Initial(J_{1})=\langle x^{11}, x^{10}z, x^{9}z^2,x^{8}z^4,x^{7}z^5,\dots,x^4z^{10},x^3z^{11},x^2z^{13},xz^{14},z^{16}\rangle
\end{align*}
and
\begin{equation*}
\Initial(J_{2})=\langle y^{11}, y^{10}z, y^9z^2,y^8z^3,y^7z^5,\dots,y^4z^9,y^3z^{10},y^2z^{11},yz^{13}, z^{14}\rangle.
\end{equation*}
Therefore,
\begin{equation*}
\Initial (J_{1}\colon J(\tau)+J_{2}\colon J(\tau))=\langle x^3, x^2z^2, y^2,yz^2,z^3\rangle.
\end{equation*}
We can see by inspection that any monomial of degree five or more is in $\langle x^3, x^2z^2, y^2,yz^2,z^3\rangle$, while $x^2yz$ is a monomial of degree four not in this ideal.  Thus $\reg R/\Initial (J_{1}\colon J(\tau)+J_{2}\colon J(\tau))=4$ and so, by Lemma~\ref{lem:presentation_H0J}, $\reg H_1(\calR_\bullet/\calJ_\bullet)=15$.
In this case, \eqref{eqn:bound_H0J} specializes to
$14\leq \reg H_{1}(\calR_{\bullet}/\calJ_{\bullet})\leq 15$. Since $r+1\equiv s-1\mod s$ and $r+1\equiv t-1\mod t$, Proposition \ref{prop:bounds_H0J} yields $\reg H_{1}(\calR_{\bullet}/\calJ_{\bullet})=15$, which aligns with what we found by inspection.
\end{Example}

\begin{proof}[Proof of Theorem~\ref{prop:nontrivialcases}]
The first equation in Theorem~\ref{prop:nontrivialcases} follows from Lemma~\ref{lem:quotientbasis}, Theorem~\ref{thm:splinedimension}, and \\
Lemma~\ref{lem:H0pres}.  The second equation follows from the first equation and Proposition~\ref{prop:2Dpolygon}.
\end{proof}

\begin{proof}[Proof of Theorem~\ref{thm:LBregularity}]
Theorem~\ref{thm:LBregularity} follows from Theorem~\ref{thm:splinedimension} and Proposition~\ref{prop:bounds_H0J}.
\end{proof}

\section{Comparison to quasi-cross-cut}\label{sec:comparison}
In this section we address the phenomenon that, for certain pairs $(r,d)$, $\dim C^{r}_{d}(\Delta)=\dim C^{r}_{d}(\Delta')$ where $\Delta'$ is obtained by removing the unique totally interior edge from $\Delta$ to get a quasi-cross-cut partition (see Definition~\ref{def:quasi-cross-cut}), as shown in Figure~\ref{fig:intrinsic_supersmoothness}.
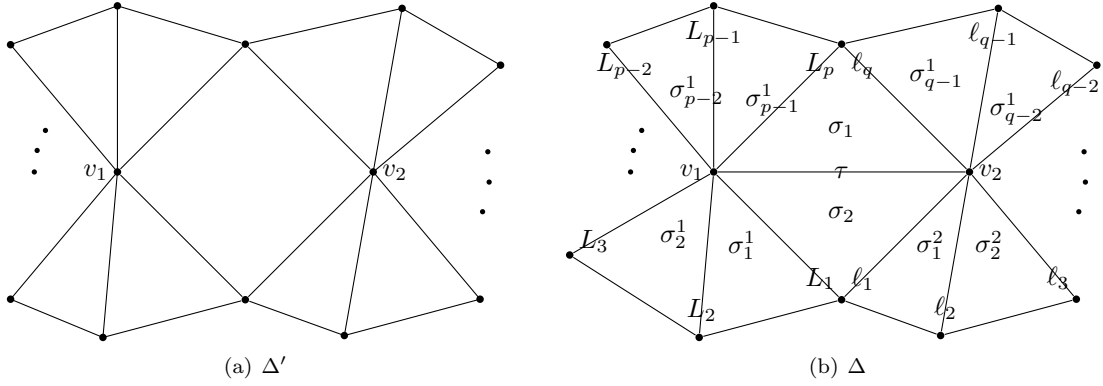
\begin{figure}
	\centering	
  \subfigure[$\Delta'$]{
\begin{tikzpicture}[scale=1.7]
\tikzstyle{dot}=[circle,fill=black,inner sep=1 pt];
\tikzstyle{smdot}=[circle,fill=black,inner sep=.7 pt];
\node[dot] (left) at (-1,0){};
\node[dot] (right) at (1,0){};
\node[dot] (up) at (0,1){};
\node[dot] (down) at (0,-1){};

\define \La {-95};
\define \Lb {-130};

\define \Lc {-230};
\define \Ld {-270};

\define \la {-100};
\define \lb {-50};

\define \lc {40};
\define \ld {80};

\define \rd {1.3};

\draw (right)node[right]{$v_2$}--(up)--(left)node[left]{$v_1$};
\draw (left)--(down)--(right);
\foreach \x in {\La,\Lb,\Lc,\Ld}
\draw (node cs:name=left)
+(0,0)--+(\x:\rd)node[dot]{};
\foreach \y in {\la,\lb,\lc,\ld}
\draw (node cs:name=right)
+(0,0)--+(\y:\rd)node[dot]{};

\foreach \source/\target in {\La/\Lb,\Lc/\Ld}
\draw (node cs:name=left) +(\source:\rd)--+(\target:\rd);
\foreach \source/\target in {\la/\lb,\lc/\ld}
\draw (node cs:name=right) +(\source:\rd)--+(\target:\rd);

\draw (-1,0)++(\Ld:\rd)--(up);
\draw (-1,0)++(\La:\rd)--(down);
\draw (1,0)++(\ld:\rd)--(up);
\draw (1,0)++(\la:\rd)--(down);

\foreach \x in {-180,-195,-210}
\draw(node cs: name=left) node[smdot] at +(\x:{.5*\rd}){};

\foreach \y in {-20,-5,10}
\draw(node cs: name=right) node[smdot] at +(\y:{.7*\rd}){};

\end{tikzpicture}

}
\hspace{10 pt}
\subfigure[$\Delta$]{
\begin{tikzpicture}[scale=1.7]
\tikzstyle{dot}=[circle,fill=black,inner sep=1 pt];
\tikzstyle{smdot}=[circle,fill=black,inner sep=.7 pt];
\node[dot] (left) at (-1,0){};
\node[dot] (right) at (1,0){};
\node[dot] (up) at (0,1){};
\node[dot] (down) at (0,-1){};

\define \rd {1.3};

\path(node cs:name=left) node[dot] (La) at +(-95:\rd){};
\path(node cs:name=left) node[dot] (Lb) at +(-150:\rd){};
\path(node cs:name=left) node[dot] (Lc) at +(-230:\rd){};
\path(node cs:name=left) node[dot] (Ld) at +(-270:\rd){};

\path(node cs:name=right) node[dot] (la) at +(-100:\rd){};
\path(node cs:name=right) node[dot] (lb) at +(-50:\rd){};
\path(node cs:name=right) node[dot] (lc) at +(40:\rd){};
\path(node cs:name=right) node[dot] (ld) at +(80:\rd){};

\define \pos {.85};
\draw (left)node[left]{$v_1$}--node{$\tau$}(right)node[right]{$v_2$}--node[pos=\pos]{$\ell_q$}(up)--node[pos=1-\pos]{$L_p$}(left);
\draw (left)--node[pos=\pos]{$L_1$}(down)--node[pos=1-\pos]{$\ell_1$}(right);
\foreach \x/\lb in {(La)/$L_2$,(Lb)/$L_3$,(Lc)/$L_{p-2}$,(Ld)/$L_{p-1}$}
\draw (left)--node[pos=\pos]{\lb}\x;
\foreach \y/\lb in {(la)/$\ell_2$,(lb)/$\ell_3$,(lc)/$\ell_{q-2}$,(ld)/$\ell_{q-1}$}
\draw (right)--node[pos=\pos]{\lb}\y;

\foreach \source/\target in {(La)/(Lb),(Lc)/(Ld)}
\draw \source--\target;
\foreach \source/\target in {(la)/(lb),(lc)/(ld)}
\draw \source--\target;

\draw (Ld)--(up);
\draw (La)--(down);
\draw (ld)--(up);
\draw (la)--(down);

\foreach \x in {-180,-195,-210}
\draw(node cs: name=left) node[smdot] at +(\x:{.5*\rd}){};

\foreach \y in {-20,-5,10}
\draw(node cs: name=right) node[smdot] at +(\y:{.7*\rd}){};

\node at (barycentric cs:left=1,right=1 ,up=1){$\sigma_1$};
\node at (barycentric cs:left=1,right=1 ,down=1){$\sigma_2$};
\node at (barycentric cs:left=2,down=1 ,La=1){$\sigma^1_1$};
\node at (barycentric cs:left=2,La=1 ,Lb=1){$\sigma^1_2$};
\node at (barycentric cs:left=12,Lc=4 ,Ld=9){$\sigma^1_{p-2}$};
\node at (barycentric cs:left=6,Ld=1 ,up=6){$\sigma^1_{p-1}$};
\node at (barycentric cs:right=2,down=1 ,la=1){$\sigma^2_1$};
\node at (barycentric cs:right=2,la=1 ,lb=1){$\sigma^2_2$};
\node at (barycentric cs:right=3,lc=2 ,ld=1){$\sigma^1_{q-2}$};
\node at (barycentric cs:right=1,ld=1 ,up=1){$\sigma^1_{q-1}$};

\end{tikzpicture}

  }
	\caption{$\Delta'$ is obtained by removing $\tau$ from $\Delta$}
	\label{fig:intrinsic_supersmoothness}
\end{figure}
In \cite{sorokina2018}, Sorokina discusses this phenomenon using the Bernstein-B\'{e}zier form in the case $s=t=2$ (which she calls the Toh\v{a}neanu partition due to its appearance in~\cite{tohaneanu2005smooth}).  A main result of~\cite{sorokina2018} is that $\dim C^{r}_{d}(\Delta)=\dim C^{r}_{d}(\Delta')$ for $d\le 2r$ when $s=t=2$.  In this section, we extend Sorokina's result to arbitrary $s$ and $t$.  This equality of dimensions upon removal of an edge is related to the phenomenon of \textit{supersmoothness}~\cite{sorokina2018,Floater-Hu-2020}, although we will not go into details about this.


Let $\Delta$ be a triangulation with a single totally interior edge $\tau$ and let $\Delta'$ be the partition formed by removing the edge $\tau$ from $\Delta$ as in Figure \ref{fig:intrinsic_supersmoothness}.  As in Section~\ref{setion:section_1tot}, put $J_{i}=\sum_{v_{i}\in\varepsilon,\varepsilon\neq\tau}J(\varepsilon)$ for $i=1,2$.  Furthermore, for any homogeneous ideal $I\subseteq R$, define
\begin{equation*}
    \initdeg~ I:=\min \{d : I_{d}\neq 0\}.
\end{equation*}

\begin{Lemma}\label{lemma:ses_intrinsic_smoothness}
There is a short exact sequence
\begin{equation}\label{eq:ses_intrinsic_supersmoothness}
    0\rightarrow C^r(\Delta')\xrightarrow{\iota} C^r(\Delta)\xrightarrow{\delta} J(\tau)\cap J_1\cap J_2 \to 0
\end{equation}
where $\iota$ is the natural inclusion and $\delta(F)=F_{\sigma_2}-F_{\sigma_1}$ is the difference of $F$ restricted to the faces $\sigma_{2}$ and $\sigma_{1}$ shown in Figure \ref{fig:intrinsic_supersmoothness}.  In particular, if $d< \initdeg~J_{1}\cap J_{2}\cap J(\tau)$ then $\dim C^{r}_{d}(\Delta)=\dim C^{r}_{d}(\Delta')$. \end{Lemma}
\begin{proof}
We prove that~\eqref{eq:ses_intrinsic_supersmoothness} is a short exact sequence; the final statement follows immediately.
Let $\tau$ be the totally interior edge of $\Delta$, with corresponding linear form $L_\tau$.  Let $L_1,\ldots,L_p$ be the linear forms defining the edges which surround the interior vertex $v_1$, in \textit{clockwise} order.  Likewise suppose that the linear forms defining the edges which surround the interior vertex $v_2$ in \textit{counterclockwise} order are $\ell_1,\ldots,\ell_q$.  See Figure~\ref{fig:intrinsic_supersmoothness}, where the edges are labeled by the corresponding linear forms.  With this convention, $J(\tau)=\langle L_\tau^{r+1}\rangle$, $J_1=\langle L_1^{r+1},\ldots,L_p^{r+1}\rangle$, and $J_2=\langle \ell_1^{r+1},\ldots,\ell_q^{r+1}\rangle$.

It is clear that $C^r(\Delta')$ is the kernel of the map $\delta$.  It follows from the algebraic spline criterion that if $F\in C^r(\Delta)$ then $\delta(F)\in J(\tau)\cap J_1\cap J_2$.  We show that $\delta$ is surjective.  Suppose that $f\in J(\tau)\cap J_1\cap J_2$.  We define a spline $F\in C^r(\Delta)$ so that $\delta(F)=f$ as follows.  Let $F_{\sigma_1}=0$ and $F_{\sigma_2}=f$.  Write $\sigma^1_1,\ldots,\sigma^1_{p-1}$ for the remaining faces surrounding the vertex $v_1$ (in clockwise order) and $\sigma^2_1,\ldots,\sigma^2_{q-1}$ for the remaining faces surrounding the vertex $v_2$ (in counterclockwise order).  See Figure~\ref{fig:intrinsic_supersmoothness}.

Then the linear forms defining the interior edges adjacent to $\sigma^1_i$ are $L_i$ and $L_{i+1}$ for $i=1,\ldots,p-1$ and the linear forms defining the interior edges adjacent to $\sigma^2_j$ are $\ell_j$ and $\ell_{j+1}$ for $j=1,\ldots, q-1$.

Now we continue to define $F$.  Since $f\in J_1, f=\sum_{i=1}^p g_iL_i^{r+1}$ for some polynomials $g_1,\ldots,g_p$.  Define $F_{\sigma^1_i}$ by $f-\sum_{j=1}^i g_jL_j^{r+1}$ for $i=1,\ldots,p-1$.  Likewise, since $f\in J_2$, $f=\sum_{i=1}^q h_i\ell_i^{r+1}$ for some polynomials $h_1,\ldots,h_q$.  Define $F_{\sigma^2_i}$ by $f-\sum_{j=1}^i h_j\ell_j^{r+1}$ for $i=1,\ldots,q-1$.  One readily checks, using Proposition~\ref{prop:algebraiccriterion}, that $F\in C^r(\Delta)$.  Clearly $\delta(F)=f$, so we are done.
\end{proof}

Since $J(\tau)$ is principal, $J_{1}\cap J_{2}\cap J(\tau)=[(J_{1}\cap J_{2}):J(\tau)]J(\tau)$ and $(J_{1}\cap J_{2}):J(\tau)=(J_{1}:J(\tau))\cap (J_{2}:J(\tau))$.
Hence, the $\initdeg(J_{1}\cap J_{2}\cap J(\tau))=(r+1)+\initdeg(J_{1}:J(\tau)\cap J_{2}:J(\tau))$.
\begin{Lemma}\label{lem:intersectiondescription}
    For $i=1,2$, let $J'_{i}=J_{i}:J(\tau)$. We have
    \begin{equation*}
        \Initial(J'_{1}\cap J'_{2})=\Initial (J'_{1})\cap\Initial(J'_{2}).
    \end{equation*}
    Moreover, the monomial $x^{A}y^{B}z^{C}\in \Initial (J'_{1})\cap\Initial(J'_{2})$ if and only if $A,B,C\geq 0$, $r+1-s< sA+(s-1)C$, and $r+1-t< tB+(t-1)C$.  
\end{Lemma}
\begin{proof}
    The second part follows immediately from Corollary \ref{cor:colontwovar}. So we only verify the first part. It is clear that $\Initial(J'_{1}\cap J'_{2})\subseteq \Initial J'_{1}\cap\Initial J'_{2}$. We only need to show that $ \dim\Initial(J'_{1}\cap J'_{2})_{d}=\dim(\Initial J'_{1}\cap\Initial J'_{2})_{d}$ for all degrees $d\geq 0$.
    
    By Corollary ~\ref{cor:initsumEsuminit}, we know that $ \dim\Initial(J'_{1}+ J'_{2})_{d}=\dim(\Initial J'_{1}+\Initial J'_{2})_{d}$. We also know that $\dim (\Initial I)_{d}=\dim I_{d}$ for any ideal $I$. Because 
    \begin{equation*}
        \dim \Initial(J'_{1}\cap J'_{2})_{d}=\dim (J'_{1})_{d}+\dim (J'_{2})_{d}-\dim (J'_{1}+ J'_{2})_{d},
    \end{equation*}
    and
    \begin{equation*}
    \dim(\Initial J'_{1}\cap\Initial J'_{2})_{d}=\dim (J'_{1})_{d}+\dim (J'_{2})_{d}-\dim (\Initial J'_{1}+ \Initial J'_{2})_{d},
    \end{equation*}
    we must have $ \dim\Initial(J'_{1}\cap J'_{2})_{d}=\dim(\Initial J'_{1}\cap\Initial J'_{2})_{d}$ for all degree $d\geq 0$. This completes the proof.
\end{proof}

\begin{Corollary}\label{cor:crosscut}
Let $\Delta$ be a triangulation with a single totally interior edge satisfying Assumptions~\ref{assumptions}.  For $d\le \frac{t}{s(t-1)}r+r$, $\dim C^{r}_{d}(\Delta)=\dim C^{r}_{d}(\Delta')$.
\end{Corollary}
\begin{proof}
As in Lemma~\ref{lem:intersectiondescription}, we let $J'_i=J_i:J(\tau)$ for $i=1,2$.  By Lemma~\ref{lemma:ses_intrinsic_smoothness}, it suffices to prove that $\frac{t}{s(t-1)}r+r<\initdeg(J_{1}\cap J_{2}\cap J(\tau))$.  From the discussion just prior to Lemma~\ref{lem:intersectiondescription} coupled with the lemma itself, it suffices to prove that $\frac{t}{s(t-1)}r-1<\initdeg(\Initial(J'_{1})\cap \Initial(J'_{2}))$.

Let $\calQ$ be the collection of points $(A,B,C)\in\RR^3$ defined by the inequalities $A,B,C\geq 0$, $r-s+1< sA+(s-1)C$ and $r-t+1< tB+(t-1)C$.  Then its closure $\overline{\calQ}$ (in the usual topology on $\RR^3$) is the polyhedron in $\RR^3$ defined by the inequalities $A,B,C\geq 0$, $r-s+1\le sA+(s-1)C$ and $r-t+1\le tB+(t-1)C$.  Using Lemma~\ref{lem:intersectiondescription} again, 
\[
\initdeg(\Initial(J'_{1})\cap \Initial(J'_{2}))=\min\{A+B+C:(A,B,C)\in\calQ\cap\ZZ^3\}.
\]
We first show that $\frac{t}{s(t-1)}r-1$ is the smallest value achieved by $A+B+C$ on the polyhedron $\overline{\calQ}$.  Since we assume that $s\leq t\leq r+1$, it is not possible for any $(A,B,C)\in \calQ$ to satisfy $A=C=0$ or $B=C=0$. The vertices of the polyhedron $\overline{\calQ}$ are:
\begin{align*}
    Q_{1}&~:~\left(\frac{t-s}{s(t-1)}r,~0,~\frac{r}{t-1}-1\right),\\
    Q_{2}&~:~\left(\frac{r+1}{s}-1,~\frac{r+1}{t}-1,~0\right),\mbox{and}\\
    Q_{3}&~:~\left(0,~0,~\frac{r}{s-1}-1\right).
\end{align*}
We have proved \eqref{eq:ineq_two_vertices_of_polytope}, which implies that $A+B+C$ evaluated at $Q_2$ is at least as large as $A+B+C$ evaluated at $Q_1$.  Since $\frac{t}{s(t-1)}r-1\leq \frac{r}{s-1}-1$, $A+B+C$ evaluated at $Q_{1}$ is at most $A+B+C$ evaluated at $Q_{3}$.  Therefore, over the real numbers, $A+B+C$ is minimized over $\overline{\calQ}$ at the vertex $Q_1$, with a value of $\frac{t}{s(t-1)}r-1$.  Let $H'$ be the affine hyperplane defined by $A+B+C=\frac{t}{s(t-1)}r-1$.  A straightforward calculation with the inequalities also shows that $H'\cap\overline{\calQ}=Q_1$, hence $H'\cap \calQ=\emptyset$ and also $H'\cap\calQ\cap\ZZ^3=\emptyset$.  It follows that 
\[
\initdeg(\Initial(J'_{1})\cap \Initial(J'_{2}))=\min\{A+B+C:(A,B,C)\in\calQ\cap\ZZ^3\}>\frac{t}{s(t-1)}r-1. \qedhere
\]
\end{proof}

\section{The explicit dimension formula}\label{sec:explicitformula}

In this section we use the preceding sections to give an explicit formula for $\dim C^r_d(\Delta)$, where $\Delta$ is a planar triangulation with a single totally interior edge, for any $d\ge 0$ and $r\ge 0$.
We then illustrate the formula in a few examples.

\begin{Theorem}\label{thm:ExplicitFormula}
Let $\Delta$ be a triangulation with a single totally interior edge $\tau$ satisfying Assumptions~\ref{assumptions} and $\Delta'$ the partition formed by removing $\tau$.  Then
\[
\dim C^r_d(\Delta)=
\begin{cases}
L(\Delta',d,r) & d\le \frac{t}{s(t-1)}r+r\\
L(\Delta,d,r)+f(\Delta,d,r) & \frac{t}{s(t-1)}r+r<d\le \frac{r+1}{s}+\frac{r+1}{t}+r-1\\
L(\Delta,d,r) & d>\frac{r+1}{s}+\frac{r+1}{t}+r-1,
\end{cases}
\]
where
\[
f(\Delta,d,r):= \sum_{i=\left\lceil \frac{2st(d-r)-(s+t)d}{(s-1)(t-1)-1}\right\rceil}^{d-r-1} \left(\left\lfloor\frac{(i-d)(s-1)}{s}+r\right\rfloor-\left\lceil\frac{i+d(t-1)}{t}-r\right\rceil+1\right).
\]
Moreover, put $\mathfrak{r}=\lfloor (r+1)/s\rfloor+\lfloor(r+1)/t\rfloor+r$.  If $t\ge 3$, $r+1\equiv s-1\mod s$, and $r+1\equiv t-1\mod t$, then $f(\Delta,\mathfrak{r},r)=1$.  Otherwise $f(\Delta,\mathfrak{r},r)=0$.
\end{Theorem}

\begin{proof}
First, it follows from $t\le r+1$ that
$
\frac{t}{s(t-1)}\le \frac{r+1}{s}+\frac{r+1}{t}+r-1.
$
Now, if $d\le \frac{t}{s(t-1)}r+r$ then $\dim C^r_d(\Delta)=\dim C^r_d(\Delta')$ by Corollary~\ref{cor:crosscut}.  Since $\Delta'$ is a quasi-cross-cut partition, it follows from~ Proposition~\ref{Prop:quasi-cross-cut} that $\dim C^r_d(\Delta')=L(\Delta',d,r)$ for all $d\ge 0$.  

Likewise, if $d>\frac{r+1}{s}+\frac{r+1}{t}+r-1$ then $\dim C^r_d(\Delta)=L(\Delta,d,r)$ by Theorem~\ref{thm:LBregularity}.  Observe that these first two cases allow us to dispense of the case $s=t=2$ (which we consider in more detail in Example~\ref{ex:s2t2}).  So henceforth we assume $t\ge 3$.

According to Theorem~\ref{prop:nontrivialcases}, it remains to show that, when $\frac{t}{s(t-1)}r+r<d\le \frac{r+1}{s}+\frac{r+1}{t}+r-1$, $f(\Delta,d,r)=\#(\calP_d\cap\ZZ^2)$, where $\calP_d$ is the polytope defined by the inequalities
$A\ge 0$, $B\ge 0$, $A\le B(s-1)-d(s-1)+sr$, $B\le A(t-1)-d(t-1)+tr$, and $A+B\le d-r-1$.

We first show that, in the given range for $d$, $\calP_d$ is in fact a triangle bounded by $A\le B(s-1)-d(s-1)+sr$, $B\le A(t-1)-d(t-1)+tr$, and $A+B\le d-r-1$.  For this observe that
\[
B\le A(t-1)-d(t-1)+tr\le B(s-1)(t-1)+sr(t-1)-d(s-1)(t-1)+tr-d(t-1)
\]
from which we deduce that $ds(t-1)-sr(t-1)-tr\le B[(s-1)(t-1)-1]$.  Since $t\ge 3$, we need only show that $0\le ds(t-1)-sr(t-1)-tr$.  Re-arranging, we see this is equivalent to 
\[
\frac{tr}{s(t-1)}+r\le d,
\]
which is precisely our assumption.  So $B\ge 0$ is a consequence of $A\le B(s-1)-d(s-1)+sr$ and $B\le A(t-1)-d(t-1)+tr$.  Since $B\ge 0$, we obtain
\[
0\le A(t-1)-d(t-1)+tr
\]
or $d(t-1)-tr\le A(t-1)$.  Using the given bound on $d$, we obtain $r(t/s)-r\le d(t-1)-tr$.  Since $s\le t$, we thus have $0\le A(t-1)$ and so $0\le A$.

It follows that $\calP_d$ is the triangle in the first quadrant bounded by $A\le B(s-1)-d(s-1)+sr$, $B\le A(t-1)-d(t-1)+tr$, and $A+B\le d-r-1$.  Now we count the lattice points $(A,B)\in\ZZ^2\cap\calP_d$.  We do this by counting the lattice points on the line segments defined by the intersection of $A+B=i$ with $\calP_d$, for $0\le i\le d-r-1$.  The two lines defined by the equations $A= B(s-1)-d(s-1)+sr$ and $B= A(t-1)-d(t-1)+tr$ intersect at the point
\[
\left( \frac{t(s-1)(d-r)-sr}{(s-1)(t-1)-1},\frac{s(t-1)(d-r)-tr}{(s-1)(t-1)-1}\right),
\]
where $A+B$ (restricted to $\calP_d$) achieves its minimum value of
\[
\frac{2st(d-r)-(s+t)d}{(s-1)(t-1)-1}.
\]
Thus we start our count at $i=\lceil(2st(d-r)-(s+t)d)/((s-1)(t-1)-1) \rceil$, which is the lower index of summation for the definition of $f(\Delta,d,r)$ in the theorem statement.  Clearly the maximum is $i=d-r-1$.

Now put $A+B=i$, so $B=i-A$.  We have
\[
A\le B(s-1)+sr-d(s-1)=(i-A)(s-1)+sr-d(s-1)
\]
yielding $sA\le (i-d)(s-1)+sr$ or $A\le (i-d)(s-1)/s+r$.  Likewise we have
\[
i-A=B\le A(t-1)+tr-d(t-1)
\]
which yields $i-tr+d(t-1)\le tA$ or $(i+d(t-1))/t-r\le A$.  Putting these together, the number of lattice points $(A,B)\in\calP_d\cap\ZZ^2$ with $A+B=i$ is the same as the number of integers $A\in\ZZ$ in the interval
\[
(i+d(t-1))/t-r\le A\le (i-d)(s-1)/s+r,
\]
which is counted by
\[
\left\lfloor\frac{(i-d)(s-1)}{s}+r\right\rfloor-\left\lceil\frac{i+d(t-1)}{t}-r\right\rceil+1.
\]
Summing this over the appropriate range for $i$ yields the expression for $f(\Delta,d,r)$.

Now put $\mathfrak{r}=\lfloor (r+1)/s\rfloor+\lfloor(r+1)/t\rfloor+r$.  If $t\ge 3$, $r+1\equiv s-1\mod s$, and $r+1\equiv t-1\mod t$ then $f(\Delta,\mathfrak{r},r)=1$ by Lemma~\ref{lem:boundC}.  Otherwise $f(\Delta,\mathfrak{r},r)=0$, also by Lemma~\ref{lem:boundC}.
\end{proof}




In the following examples we compute explicit formulas for certain triangulations with a single totally interior edge.  We assume that the triangulation $\Delta$ satisfies Assumptions~\ref{assumptions} and we introduce some additional notation to explicitly write out Schumaker's lower bound.
 Let $\alpha_1$ and $\nu_1$ (respectively $\alpha_2$ and $\nu_2$) be the quotient and remainder when $(s+1)(r+1)$ is divided by $s$ (respectively $(t+1)(r+1)$ is divided by $t$).  That is, $(s+1)(r+1)=\alpha_1s+\nu_1$ and $(t+1)(r+1)=\alpha_2t+\nu_2$, where $0\le \nu_1<s$ and $0\le \nu_2<t$.  Furthermore, put $\mu_1=s-\alpha_1$ and $\mu_2=t-\alpha_2$.  From Proposition~\ref{prop:SchumakerLower} we have
\begin{equation}\label{eq:SchumakerLowerDelta}
L(\Delta,d,r):= \binom{d+2}{2}+(p-s+q-t-1)\binom{d+1-r}{2}+\sum_{i=1,2} \mu_i\binom{d+2-\alpha_i}{2}+\nu_i\binom{d+1-\alpha_i}{2}.
\end{equation}
Now let $\alpha'_1$ and $\nu'_1$ (respectively $\alpha'_2$ and $\nu'_2$) be the quotient and remainder when $s(r+1)$ is divided by $s-1$ (respectively $t(r+1)$ is divided by $t-1$).  That is, $s(r+1)=\alpha'_1(s-1)+\nu'_1$ and $t(r+1)=\alpha'_2(t-1)+\nu'_2$, where $0\le \nu'_1<s-1$ and $0\le \nu'_2<t-1$.  Furthermore, put $\mu'_1=s-1-\alpha'_1$ and $\mu'_2=t-1-\alpha'_2$.  Again from Proposition~\ref{prop:SchumakerLower} we have\begin{equation}\label{eq:SchumakerLowerDeltaPrime}
L(\Delta',d,r):= \binom{d+2}{2}+(p-s+q-t)\binom{d+1-r}{2}+\sum_{i=1,2} \mu'_i\binom{d+2-\alpha'_i}{2}+\nu'_i\binom{d+1-\alpha'_i}{2}.
\end{equation}

\begin{Example}\label{ex:s2t2}
Consider the triangulation with $p=q=4$ and $s=t=2$.  Sorokina calls this a Toh\v{a}neanu partition in~\cite{sorokina2018} due to its study by Toh\v{a}neanu in~\cite{tohaneanu2005smooth} and \cite{Minac-Tohaneanu-2013}.  According to Theorem~\ref{thm:ExplicitFormula}, we have
\[
\dim C^r_d(\Delta)=
\begin{cases}
L(\Delta',d,r) & d\le 2r\\
L(\Delta,d,r) & d\ge 2r+1.
\end{cases}
\]
The first case (for $d\le 2r$) recovers~\cite[Theorem~3.1]{sorokina2018}.  The second case (for $d\ge 2r+1$) recovers the main result of~\cite{Minac-Tohaneanu-2013}.  From Equation~\eqref{eq:SchumakerLowerDeltaPrime}  we have
\[
L(\Delta',d,r)=\binom{d+2}{2}+4\binom{d+1-r}{2}+2\binom{d-2r}{2}.
\]
Since the final term of $L(\Delta',d,r)$ vanishes for $d\le 2r$, we have $\dim C^r_d(\Delta)=\binom{d+2}{2}+4\binom{d+1-r}{2}$ for $d\le 2r$, recovering~\cite[Theorem~3.2]{sorokina2018}.  When $r=2k-1$ is odd ($k\ge 1$), Equation~\eqref{eq:SchumakerLowerDelta} yields
\[
L(\Delta,d,2k-1)=\binom{d+2}{2}+3\binom{d+1-2k}{2}+2\binom{d-3k}{2}+2\binom{d+1-3k}{2}
\]
and when $r=2k$ is even ($k\ge 0$), Equation~\eqref{eq:SchumakerLowerDelta} yields
\[
L(\Delta,d,2k)=\binom{d+2}{2}+3\binom{d+2-2k}{2}+4\binom{d+2-3k}{2}.
\]
If $r=2k-1$ ($k\ge 1$), then 
\[
L(\Delta',4k-2,2k-1)-L(\Delta,4k-2,2k-1)=\binom{2k}{2}-4\binom{k}{2}=k>0
\]
and if $r=2k$ ($k\ge 1$), then
\[
L(\Delta',4k,2k)-L(\Delta,4k,2k)=\binom{2k+1}{2}-2\binom{k+1}{2}-2\binom{k}{2}=k>0.
\]
This proves that $\dim C^r_{2r}(\Delta)>L(\Delta,2r,r)$, so Schumaker's lower bound indeed does not give the correct dimension for $d\le 2r$ (the inequality also follows from an application of Theorem~\ref{thm:LBregularity}).  This recovers the main result of~\cite{tohaneanu2005smooth}.
\end{Example}

\begin{Example}\label{ex:explicits3t4}
Consider the triangulation $\Delta$ shown in Figure~\ref{fig:s3t4}, with $p=6,s=3,q=5,$ and $t=4$.  If $r\le 5$ then $C^r(\Delta)$ is free and $\dim C^r_d(\Delta)=L(\Delta,d,r)$ for all integers $d\ge 0$ by Proposition~\ref{prop:trivialcases}.  For $r\ge 6$, according to Theorem~\ref{thm:ExplicitFormula}, $\dim C^r_d(\Delta)=L(\Delta',d,r)$ for $d\le 13r/9$ and $\dim C^r_d(\Delta)=L(\Delta,d,r)$ for $d>(19r-5)/12$. For $13r/9<d\le (19r-5)/12$,
\begin{align*}
\dim C^r_d(\Delta)= & L(\Delta,d,r)+f(\Delta,d,r)\\
=& L(\Delta,d,r)+\sum_{i=\lceil(17d-24r)/5\rceil}^{d-r-1} \left(\lfloor 2/3(i-d)+r\rfloor-\lceil (i+3d)/4-r\rceil+1\right).
\end{align*}
We now use Equations~\eqref{eq:SchumakerLowerDelta} and~\eqref{eq:SchumakerLowerDeltaPrime} to compute dimension formulas for $r=6,7,8$.  When $r=6$, we have
\[
\dim C^6_d(\Delta)=
\begin{cases}
L(\Delta',d,6)= \binom{d+2}{2}+4\binom{d-5}{2}+2\binom{d-7}{2}+2\binom{d-8}{2}+\binom{d-9}{2} & d\le 8\\
L(\Delta,d,6)= \binom{d+2}{2}+3\binom{d-5}{2}+\binom{d-6}{2}+5\binom{d-7}{2}+\binom{d-8}{2} & d\ge 10
\end{cases}.
\]
When $d=9$,
\[
f(\Delta,d,r)=\sum_{i=2}^{2} \left(\lfloor 2/3(i-9)+6\rfloor-\lceil (i+3\cdot 9)/4-6\rceil+1\right)=0,
\]
which simply means that the triangle defined by the inequalities in Proposition~\ref{prop:2Dpolygon} does not contain any lattice points. Thus $\dim C^6_9(\Delta)=L(\Delta,9,6)$.  This also is expected by Theorem~\ref{thm:ExplicitFormula} since $\mathfrak{r}=\lfloor 7/3\rfloor+\lfloor 7/4\rfloor+6=9$ and $r+1=7\not\equiv 2\mod 3$.

Observing that the last three terms of $L(\Delta',d,r)$ vanish when $d\le 8$, we conclude that
\[
\dim C^6_d(\Delta)=
\begin{cases}
\binom{d+2}{2}+4\binom{d-5}{2} & d\le 8\\
\binom{d+2}{2}+3\binom{d-5}{2}+\binom{d-6}{2}+5\binom{d-7}{2}+\binom{d-8}{2} & d\ge 9
\end{cases}.
\]

When $r=7$, there is no integer $d$ so that $91/9=13r/9<d\le (19r-5)/12=128/12$, so we simply have
\[
\dim C^7_d(\Delta)=
\begin{cases}
L(\Delta',d,7)=\binom{d+2}{2}+4\binom{d-6}{2}+\binom{d-8}{2}+2\binom{d-9}{2}+2\binom{d-10}{2} & d\le 10\\
L(\Delta,d,7)=\binom{d+2}{2}+3\binom{d-6}{2}+5\binom{d-8}{2}+2\binom{d-9}{2} & d\ge 11
\end{cases}
\]

When $r=8$, we hit our first non-zero contribution from $f(\Delta,d,r)$.  Namely, when $d=12$, $f(\Delta,12,8)=1$ (this comes from the single lattice point pictured on the right in Figure~\ref{fig:3Dpolytope}).  Notice that $\mathfrak{r}=\lfloor 9/3\rfloor+\lfloor 9/4\rfloor+8=13$, so we must compute $f(\Delta,12,8)$ directly.  Thus
\[
\dim C^8_d(\Delta)=
\begin{cases}
L(\Delta',d,r)=\binom{d+2}{2}+4\binom{d-7}{2}+3\binom{d-10}{2}+\binom{d-11}{2}+\binom{d-12}{2} & d\le 11\\
L(\Delta,12,r)+f(\Delta,12,r)=134+1=135 & d=12\\
L(\Delta,d,r)=\binom{d+2}{2}+3\binom{d-7}{2}+3\binom{d-9}{2}+4\binom{d-10}{2} & d\ge 13
\end{cases}
\]

\end{Example}

\begin{Remark}
The smallest values of $s$, $t$, $r$, and $d$ where we see a non-zero contribution from $f(\Delta,d,r)$ in Theorem~\ref{thm:ExplicitFormula} are $s=2,t=3,r=5$, and $d=9$, where $f(\Delta,9,5)=1$.
\end{Remark}

\section{Concluding remarks and open problems}\label{sec:conclusions}
We close with a number of remarks on connections to the literature and open problems.

\begin{Remark}
It should be possible to use our techniques to analyze additional `supersmoothness' across the totally interior edge, as Sorokina does in~\cite{sorokina2018}.
\end{Remark}

\begin{Remark}
We can apply the methods of this paper to determine $\dim C^r_d(\Delta)$ whenever the only non-trivial generators of $H_1(\calR_\bullet/\calJ_\bullet)$ correspond to totally interior edges which do not meet each other.  In this case the dimension of $H_1(\calR_\bullet/\calJ_\bullet)_d$ would be obtained by simply adding together the contributions from the different totally interior edges.
\end{Remark}

\begin{Remark}
The next natural case in which to compute $\dim C^r_d(\Delta)$ for all $d,r\ge 0$ is the case when $\Delta$ has two totally interior edges which meet at a vertex.
\end{Remark}

\begin{Problem}
Suppose $\Delta$ is a triangulation with two totally interior edges which meet at a common vertex.  Find a formula for $\dim C^r_d(\Delta)$ for all $r,d\ge 0$, and all choices of vertex coordinates.
\end{Problem}

The recent counterexample to Schenck's `$2r+1$' conjecture in~\cite{Yuan-Stillman-2019,Schenck-Stillman-Yuan-2020} is a triangulation with two totally interior edges which meet at a vertex.  Thus, contrary to our result for triangulations with a single totally interior edge in Corollary~\ref{cor:2r+1}, we might not have $\dim C^r_d(\Delta)=L(\Delta,d,r)$ for $d\ge 2r+1$ when $\Delta$ is a triangulation with two totally interior edges meeting at a vertex.

\begin{Remark}
The well-known Morgan-Scott split, for which $\dim C^r_d(\Delta)$ depends on the global geometry of $\Delta$, has three totally interior edges which form a triangle.  In a remarkable preprint, Whiteley shows that the process of \textit{vertex splitting} applied to the Morgan-Scott split leads to infinitely many triangulations for which the dimension of $C^1$ quadratic splines depends on global geometry~\cite{WhiteleyBivariateGeometry}.  Vertex splitting results in a triangulation with additional triangles all of whose edges are totally interior edges.  As far as we are aware, the Morgan-Scott split and its vertex splits are the only known triangulations for which the dimension of $C^1$ quadratic splines exhibits a dependence upon global, as oppoosed to local, geometry.  For each $r\ge 2$, there is a variation $\Delta_{MS}^r$ of the Morgan-Scott split so that $C^r_{r+1}(\Delta_{MS}^r)$ exhibits dependence on global geometry~\cite{Luo-Liu-Shi-2010}.  Each of these has three totally interior edges forming a triangle as well.  Given that Theorem~\ref{thm:ExplicitFormula} implies that a triangulation with a single totally interior edge depends only on local geometry, we pose Problem~\ref{prob:localgeom}.
\end{Remark}

\begin{Problem}\label{prob:localgeom}
If no triangle of $\Delta$ is surrounded by totally interior edges -- equivalently, the dual graph has no interior vertex -- does the dimension of $C^r_d(\Delta)$ depend only on local geometry (that is, the number of slopes meeting at each interior vertex)?
\end{Problem}

\begin{Remark}
If $\Delta$ is a rectilinear partition, a \textit{mixed} spline space on $\Delta$, written $C^\alpha(\Delta)$, is one where different orders of smoothness are imposed across different edges according to a function $\alpha:\Delta^\circ_1\to\ZZ_{\ge 0}$.  Generally speaking, decreasing the order of smoothness across certain edges of a partition enriches the resulting spline space, while increasing the smoothness coarsens the spline space.

In~\cite{DiP-2018} it is shown that the (Castelnuovo-Mumford) regularity of the mixed spline space $C^\alpha(\Delta)$ on a rectilinear partition $\Delta$ can be bounded by the maximum regularity of the space of mixed splines on the union of two adjacent polygonal cells of $\Delta$ --that is, the \textit{star} of an edge --where vanishing is imposed (to the order prescribed by $\alpha$) across all edges which the polygonal cells do not have in common.  It may be possible that the methods of this paper can be used to improve the regularity bounds derived in~\cite{DiP-2018} for mixed splines on the star of an edge with vanishing imposed across the boundary.  Improving the regularity bound for splines on the star of an edge with vanishing across the boundary will give a better bound on the degree $d$ needed for the formula $\dim C^\alpha_d(\Delta)$ to stabilize.
\end{Remark}

\begin{Remark}\label{rem:generalizedqccp}
A \textit{generalized} quasi-cross-cut partition (see~\cite{Manni-1992}) is defined as follows.  We call a sequence of adjacent edges of $\Delta$ a \textit{cross-cut} if they all have the same slope and both endpoints of the sequence touch the boundary of $\Delta$.  We call a sequence of adjacent edges of $\Delta$ a \textit{quasi-cross-cut} if all edges have the same slope, one endpoint of the sequence touches the boundary, and the other endpoint cannot be extended to include another adjacent edge of the same slope.  It is possible that a cross-cut or quasi-cross-cut consists of only a single edge -- for instance, any edge which is not totally interior is either a quasi-cross-cut or it can be extended to a quasi-cross-cut.  For a vertex $\gamma\in\Delta^\circ_0$, we define $C_\gamma$ to be the number of cross-cuts passing through $\gamma$ and $F_\gamma$ to be the number of quasi-cross-cuts passing through $\gamma$.  The rectilinear partition $\Delta$ is a generalized quasi-cross-cut partition if $C_\gamma+F_\gamma\ge 2$ for every $\gamma\in\Delta^\circ_0$.  Generalized quasi-cross-cut partitions are studied by Manni in~\cite{Manni-1992} and Shi, Wang, and Yin in~\cite{Shi-Wang-Yin-1998}.

If $\Delta$ has a single totally interior edge, it is clearly a generalized quasi-cross-cut partition.
If $\Delta$ has a single totally interior edge connecting vertices $v_1$ and $v_2$ with $s+1$ different slopes meeting at $v_1$ and $t+1$ different slopes meeting at $v_2$, it follows from~\cite[Theorem~2.2]{Manni-1992} that $\dim C^r_d(\Delta)=L(\Delta,r,d)$ for $d\ge r+1+2\lceil (r+1)/(s-1)\rceil$ and from~\cite[Theorem~5]{Shi-Wang-Yin-1998} that $C^r_d(\Delta)=L(\Delta,r,d)$ for $d\ge r+\lfloor r/(s-1)\rfloor+\lfloor r/(t-1)\rfloor$.  Theorem~\ref{thm:LBregularity} shows an improvement on both of these bounds.  This leads us to pose Problem~\ref{prob:qcc}, inspired by the result of Shi, Wang, and Yin and our Theorem~\ref{thm:LBregularity}.
\end{Remark}

\begin{Problem}\label{prob:qcc}
If $\Delta$ is a generalized quasi-cross-cut partition, define for each edge $\tau=\{u,v\}\in\Delta_1$ the quantity $\xi_\tau=(r+1)/(C_u+N_u)+(r+1)/(C_v+N_v)$.  Let $\xi_{\Delta}=\max\{\xi_\tau:\tau\in\Delta^\circ_1\}$.  Is it true that
$
\dim C^r_d(\Delta)=L(\Delta,d,r)
$
for $d>\xi_\Delta+r-1$?
\end{Problem}

If Problem~\ref{prob:qcc} has a positive answer, it would imply that all generalized quasi-cross-cut partitions satisfy Schenck's `$2r+1$' conjecture (and the conjecture of Alfeld and Manni for $\dim C^1_3(\Delta)$).  Notice that the only known counterexample to Schenck's conjecture in~\cite{Yuan-Stillman-2019,Schenck-Stillman-Yuan-2020} is not a generalized quasi-cross-cut partition since the central vertex has only a single cross-cut passing through it, and no quasi-cross-cuts.

\section*{Acknowledgements}
The core ideas that led to this paper can be found in a preprint of the second author: \textit{An upper bound on the regularity of the first homology of spline complexes} \href{https://arxiv.org/pdf/1907.10811.pdf}{https://arxiv.org/pdf/1907.10811.pdf}.
The current paper extends these ideas to give a complete dimension formula.  The authors participate in a semi-regular virtual meeting on topics related to splines along with Peter Alfeld, Tatyana Sorokina, Nelly Villamizar, Maritza Sirvent, and Walter Whiteley.  This meeting facilitated our collaboration, and we are grateful to all members of this group for their comments on this work and many inspiring discussions.

%


\end{document}